\documentclass[10pt,svgnames]{amsart}
\usepackage{amsfonts, amssymb, amsmath, amsthm}
\usepackage{mathrsfs} 
\usepackage{latexsym}
\usepackage{enumerate}
\usepackage{multicol}
\usepackage{verbatim}
\usepackage{url}
\usepackage{csquotes}
\usepackage{placeins}
\usepackage{epic}
\usepackage{graphicx}
\usepackage{epstopdf}
\usepackage{pstricks}
\usepackage{pst-plot}
  \usepackage{tikz-cd}
\usepackage[colorlinks=true]{hyperref}
\hypersetup{citecolor=blue, linkcolor=blue, urlcolor  = blue}

\input amssym.def
\usepackage{amscd}
\usepackage[mathscr]{eucal}
\usepackage{palatino}
\usepackage{mathtools}

\newfont{\cyrr}{wncyr10}
\usepackage{tikz-cd}
\usepackage[utf8]{inputenc}

\usepackage{mathtools}
\usepackage{graphicx}
\usepackage{pgfplots}

\newcommand{\N}{{\mathfrak N}}
\newcommand{\Z}{{\mathbb Z}}
\newcommand{\Q}{{\mathbb Q}}

\newcommand{\C}{{\mathbb C}}

\newcommand{\K}{{\mathbf K}}
\renewcommand{\P}{{\mathcal P}}
\newcommand{\p}{{\mathfrak p}}
\newcommand{\q}{{\mathfrak q}}
\newcommand{\W}{{\mathcal W}}
\renewcommand{\O}{{\mathcal O}}
\renewcommand{\L}{{\mathbf L}}

\renewcommand{\a}{{\mathfrak  a}}
\renewcommand{\b}{{\mathfrak  b}}
\renewcommand{\c}{{\mathfrak  c}}
\newcommand{\B}{{\mathfrak  B}}

\newcommand{\e}{{\mathfrak  e}}
\newcommand{\m}{{\mathfrak  m}}
\newcommand{\M}{{\mathbf  M}}
\newcommand{\rO}{{\rm O}}
\renewcommand{\d}{{\mathfrak{d}}}
\newcommand{\F}{{\mathbb F}}

\newcommand{\psum}{\sideset{}{^*}\sum}

\newtheorem{thm}{Theorem}

\newtheorem{lem}[thm]{Lemma}
\newtheorem{cor}[thm]{Corollary}
\newtheorem{prop}[thm]{Proposition}

\newtheorem{rmk}[thm]{Remark}
\newtheorem{defn}[thm]{Definition}

\newcommand{\thmref}[1]{Theorem~\ref{#1}}
\newcommand{\corref}[1]{Corollary~\ref{#1}}

\newcommand{\lemref}[1]{Lemma~\ref{#1}}

\def\ndiv{{ | \negthinspace \negthinspace / }}
\def\ndiv{{\kern3pt | \kern-6pt - \kern1pt}}
\DeclareMathOperator{\gal}{Gal}

\begin{document}

\title[On Gerth's heuristics for quadratic extensions of certain number fields]{On Gerth's heuristics for a family of quadratic extensions of certain Galois number fields}

\author{ C. G. K. Babu, R. Bera}

\address[C. G. K. Babu, R. Bera]
{ Indian Statistical Institute, 8th Mile, Mysore Rd, RVCE Post, Bengaluru, Karnataka 560059.}

\author{ J. Sivaraman}

\address[J. Sivaraman]
{IISER Thiruvananthapuram Campus, Maruthamala P. O, Vithura, Kerala 695551.}

\author{ B. Sury}

\address[B. Sury]
{ Indian Statistical Institute, 8th Mile, Mysore Rd, RVCE Post, Bengaluru, Karnataka 560059.}

\maketitle

\begin{abstract}
Gerth generalised Cohen-Lenstra heuristics to the prime $p=2$.
He conjectured that for any positive integer $m$, the limit
 $$
 \lim_{x \to \infty} \frac{\sum_{0 < D \le X, \atop{ \text{squarefree} }} |{\rm Cl}^2_{\Q(\sqrt{D})}/{\rm Cl}^4_{\Q(\sqrt{D})}|^m}{\sum_{0 < D \le X, \atop{ \text{squarefree} }} 1}
 $$
 exists and proposed a value for the limit.
Gerth's conjecture was proved by Fouvry and Kluners in 2007. In this paper, we generalize their result by obtaining lower
bounds for the average value of  $|{\rm Cl}^2_{\L}/{\rm Cl}^4_{\L}|^m$, where
$\L$ varies over an infinite family of quadratic extensions of certain Galois number fields. As a special case of our theorem we obtain
lower bounds for the average value when the base field is any Galois number field
with class number $1$ in which $2\Z$ splits.
\end{abstract}

\section{Introduction}
In 1984, based on some numerical evidence, Cohen and Lenstra made
striking conjectures on the structure of the odd part of the
(narrow) class group, and on divisibility properties for class
numbers of quadratic fields. For instance, their predictions imply
that, for any positive integer $n$, quadratic fields with class
number divisible by $n$, must have positive density among the
family of all quadratic fields. Among other things, Cohen-Lenstra's
conjecture asserts that the probability that the class number of a
real quadratic field is divisible by an odd prime $p$ is
$$1 -
\prod_{n=2}^{\infty} (1- 1/p^n).
$$
A key idea of Cohen-Lenstra is to
associate as a weight to the class group, the reciprocal of the
order of its automorphism group. In 1987, F. Gerth modified the
Cohen-Lenstra heuristics to the prime $p=2$ by considering the
square of the class group. Twenty years later, E. Fouvry and J. Kluners \cite{FK07} confirmed these predictions
on the $4$-ranks of class groups ${\rm Cl}_\K$ of quadratic fields $\K$.
Here, $4$-rank refers to the $\mathbb{F}_2$-dimension of
${\rm Cl}_\K^2/{\rm Cl}_\K^4$. If $f$ is a sufficiently nice, real-valued, positive
function on the set of discriminants, one can define its average
value in a natural manner. When $f$ is the characteristic function
of a set of discriminants satisfying some specific property, this
average value - if it exists - is said to be the density of this set
of discriminants. If $f(D)$ is of the form
$$\prod_{i=0}^r
\left(2^{{\rm dim}_{\mathbb{F}_2}\left({\rm Cl}_{\Q(\sqrt{D})}^2/{\rm Cl}_{\Q(\sqrt{D})}^4\right)} - 2^i\right)
$$
for some positive
integer $r$, Fouvry-Kluners obtain the densities both for positive
as well as negative fundamental discriminants $D$, thereby
confirming Gerth's conjecture. \vskip 5mm

\noindent In this paper, we follow their technique to treat
the more general case of quadratic extensions of a class of number
fields that have certain nice properties. We obtain lower bounds for
the densities. The lower bound involves the number of subspaces of
cardinality $2^m$ in $\mathbb{F}_2^{2m}$. This idea of employing the
geometry of $\mathbb{F}_2$-vector spaces was a novel one introduced
by Fouvry and Kluners. In order to use their ideas, we restrict our
base field to a family of class number one fields. Even though the outline of our proofs
follows that of Fouvry and Kluners, we need to carry out a number of
technical generalizations to adapt their proof. We
use generalized versions of the Hilbert and Jacobi symbols, and of
the Siegel-Walfisz theorem and other analytic estimates to complete our proof.
Throughout this article, $\K$ will be a number field such that:
\begin{enumerate}
\item the extension $\K/\Q$ is Galois,\label{one}
\item the ring of integers $\O_{\K}$ is a principal ideal domain, and \label{two}
\item there exists a unit $\varepsilon \in \O_{\K}^*$ such that the order of $\varepsilon \bmod \p^2$ is $2$ for all $\p \mid 2\O_{\K}$. \label{units}.
\end{enumerate}
Examples of quadratic fields satisfying the above conditions are $\Q$, $\Q(i)$ and $\Q(\sqrt{3})$ to name a few.
Further, we show in Section \ref{examples} that any Galois number field with class number $1$ in which $2\Z$ splits will also satisfy the
above three conditions.
 We refer the reader Section \ref{examples} for several explicit examples.
For a number field $\K$, we shall use $n_\K$, $r_{\K}$ and $\N$ to denote, respectively, its degree, rank of the unit group, and the norm map to $\Q$.
Further, for a field extension $\L$ of $\K$,  ${\rm Cl}_{\L}$, will denote the class group of $\L$ and
${\rm rk}_4({\rm C}_{\L})$ will be used to denote
the 4-rank of $\L$, viz.
$$
{\rm dim}_{\mathbb{F}_2}\left({\rm Cl}_{\L}^2/ {\rm Cl}_{\L}^4\right).
$$
Our aim is to obtain, for a family $\mathcal{F}$ of quadratic extensions of $\K$ and any positive integer $m$,
a non-trivial lower bound for
$$
\liminf_{X \to \infty } \frac{\sum_{\L \in \mathcal{F}(X)} 2^{m \cdot {\rm rk}_4({\rm Cl}_{\L})}}{\sum_{\L \in \mathcal{F}(X)} 1}
$$
where $\mathcal{F}(X)$ is used to denote the number of fields in $\mathcal{F}$ for which the absolute norm of relative discriminant of $\L/\K$
at most $X$.
\subsection*{Choosing the family $\mathcal{F}$ of quadratic extensions of $\K$}
We state first a lemma due to Smith \cite{HS} to help us in
selecting an appropriate family of quadratic extensions of $\K$.
\begin{lem}[H. Smith \cite{HS}]\label{mono}
Let $\p$ be  a prime above $2\Z$ in $\K$.
We use $f$ to denote the residue class degree of $\p$ over $\Q$.
Let $\L = \K(\sqrt{\alpha})$ for some $\alpha \in \O_{\K}$. Then $\O_{\L} = \O_{\K}[\sqrt{\alpha}]$
if and only if $\alpha\O_{\K}$ is square free and
$$
\alpha^{2^f} \not \equiv \alpha \bmod \p^2 \phantom{mm} \text{for all } \p \mid 2\O_{\K}.
$$
\end{lem}
Let $\M$ be the compositum of $\K_{4\O_{\K}}$ and $\K((\O_{\K}^*)^{1/2})$ and let $\mathfrak{f}$ be the conductor of $\M/\K$.
Here $\K_{4\O_{\K}}$ (resp. $\K_{\mathfrak{f}}$) is the ray class field of $\K$ with respect to the modulus $4\O_{\K}$ (resp. $\mathfrak{f}$).
We will now define a family of quadratic extensions of
such a field $\K$.
\begin{defn}\label{oneS}
Let $\zeta_j$ be a generator of subgroup of the roots of unity in $K$ and let $S = \{\varepsilon_1, \ldots \varepsilon_{r}\}$ be a set of fundamental units
generating $\O_{\K}^*$ modulo its torsion part.
We denote by $\mathcal{C}$ the product of the absolute norms of the conductors of the orders $\O_{\K}[\sqrt{\varepsilon}]$ in $\O_{\K(\sqrt{\varepsilon})}$
as we vary $\varepsilon$ in the set $S \cup \{\zeta_j\}$.
Let
$$
\P_{\K} = \{ \p \subset \O_{\K} : (\p, \mathcal{C}\O_{\K}) = \O_{\K} \text{ and } \p \text{ splits in } \K_{\mathfrak{f}}\}.
$$
\end{defn}
\begin{defn}\label{clas of alphas defn}
Let
\begin{eqnarray}
\W & = & \{  \alpha \O_{\K} \subset \O_{\K} : \p \mid \alpha\O_{\K} \implies \p \in \P_{\K}, \alpha\O_{\K} \text{ square free } \}
\text{ and } \\
\W(X) & = & \{ \alpha \O_{\K} \subset \O_{\K} :  \N(\alpha \O_{\K}) \le X, ~ \alpha \O_{\K} \mid P_{\K}(X) \}.
\end{eqnarray}
For a generator $\alpha$ of $\alpha\O_{\K}$ such that
\begin{eqnarray}\label{cond2}
\alpha^{2^f} \not \equiv \alpha \bmod \p^2 \phantom{m} \text{for all } \p \mid 2\O_{\K},
\end{eqnarray}
we set $\L_{\alpha}$ to be $\K(\sqrt{\alpha})$.
\end{defn}
It stands to question why an $\alpha$ satisfying \eqref{cond2} should exist for $\alpha\O_{\K} \in \W$. To see this, we note
that by condition \ref{units} there exists a unit $\varepsilon \in \O_{\K}^*$
such that order of $\varepsilon \bmod \p$ is $2$ for all $\p \mid 2\O_{\K}$. If there exists a generator $\alpha$ for the ideal $\alpha\O_{\K} \in \W$
satisfying \eqref{cond2}, we are done.  If not, the generator $\varepsilon \alpha $ will
satisfy \eqref{cond2}.

Finally we define
\begin{eqnarray*}
\mathcal{F'} &:= & \{\L_{\alpha} : \alpha\O_{\K} \in \mathcal{W}\}.
\end{eqnarray*}
For each such $\L_{\alpha} \in \mathcal{F}'$,  by \lemref{mono} we now have $\O_{\L_{\alpha}} = \O_{\K}[\sqrt{\alpha}]$. Since $\{ 1, \sqrt{\alpha} \}$ is a relative integral basis of $\O_{\L_{\alpha}}/\O_{\K}$,
we know that $\mathfrak{d}_{\L_{\alpha}/\K} = 4\alpha \O_{\K}$. We now choose a set $\mathcal{F} \subset \mathcal{F}' $ such that
\begin{itemize}
\item for any $\alpha\O_{\K} \in \mathcal{W}$, $\L_{\alpha_1} \in \mathcal{F}$ for exactly one generator $\alpha_1$ of $\alpha\O_{\K}$ and
\end{itemize}
Let us also set for convenience the following notation:
$$
\mathcal{F}(X) := \{ \L \in \mathcal{F} : \N(\mathfrak{d}_{\L/\K}) \le 4X\}.
$$
We state the main theorem of our article now.
\begin{thm}\label{mainthm}
For any positive integer $m$ and $x$ varying in $\mathbb{R}$,
$$
\liminf_{X \to \infty} \frac{\sum_{\L \in \mathcal{F}(X)} 2^{m\cdot rk_4({\rm Cl}_{\L}) }}{\sum_{\L \in \mathcal{F}(X)}1} \ge  \frac{\mathcal{N}(2m,2)}{2^{m(r_\K+1)}}.
$$
Here $\mathcal{N}(2m,2)$ is used to denote the number of subspaces of $\mathbb{F}_2^{2m}$ of $2^m$ elements.

\end{thm}
Infinitude of the family $\mathcal{F}$ follows from Chebotarev's density theorem. In Section \ref{infinitude},
we compute an asymptotic for the cardinality of the set $\mathcal{F}(X)$ as $X \to \infty$. In order to prove our theorem, an important
ingredient is a lower bound for $2^{ rk_4({\rm Cl}_{\L}) }$ for  $\L \in \mathcal{F}$. This is computed in Section \ref{lowerbound}.
In Section \ref{algprelims} we recall the definition of Ray class groups and introduce an analogue of the Jacobi symbol which will play the main role in the proof of \thmref{mainthm}.
We also prove certain properties of this new symbol.
Section \ref{anaprelims} has been divided into two parts.
Subsection \ref{anaprelims1} deals with the divisor function, some of its variants and their average orders. In Subsection \ref{anaprelims2} we state some important character sum results such as the Large Sieve inequality for number fields, Siegel-Walfisz for number fields and prove a generalisation of a Lemma of Heilbronn. In Section \ref{avg} we begin computing a lower bound for the average of $2^{m \cdot {\rm rk}_4({\rm Cl}_{\L})}$ as $\L$ varies in $\mathcal{F}$, for any positive integer $m$.
This section is divided into several parts. Subsections \ref{largenumofdiv} and \ref{fam4} are devoted to rewriting the main sum and  bounding the contribution of certain subsums, culminating in
Propostion \ref{4family prop}. In Subsection \ref{gui} we state a result on the indices of this new sum which are not ``linked" (see Section \ref{avg} for definition). This will be used  to rewrite the main sum of Propostion \ref{4family prop}  in Subsection \ref{final}. We complete the proof of \thmref{mainthm} in Subsection \ref{final}. Finally, we conclude with examples of fields $\K$ which satisfy conditions  \ref{one}, \ref{two} and \ref{units} in Section \ref{examples}.
\section{Cardinality of  $\mathcal{F}(X)$}\label{infinitude}

We begin by noting that
\begin{eqnarray*}
\#\mathcal{F}(X) =
\#\{ \alpha\O_{\K} :\alpha\O_{\K}\in \W(X) \}
 =
\sum_{ \N(\alpha \O_{\K}) \le X, \atop \alpha \O_{\K} \mid P_{\K}(X)} 1
\end{eqnarray*}
We recall here a version of the Tauberian theorem as seen in \cite{RM}.
Let us define the dirichlet series
\begin{eqnarray}\label{f}
f(s) ~=~ \sum_{\a\neq (0) \atop \a \subset \O_{\K}} \frac{b_{\a}}{\N\a^s}~ = ~\prod_{\p \in \P_{\K}} \left( 1 + \frac{1}{\N\p^s} \right) = \sum_{n>0} \frac{a_n}{n^s}.
\end{eqnarray}
Note that $b_{\a}=1$ if $\a$ is squarefree and composed only of the primes in $\P_{\K}$. Further $b_\a$ is $0$ otherwise
and $a_n = \sum_{\N \a=n} b_{\a} $.
\begin{thm}
Let $0 < \alpha < 1$ be a real number.
Suppose that we can write
$$
f(s) = \frac{h(s)}{(s-1)^{1-\alpha}}
$$
for some $h(s)$ holomorphic in $\Re(s) \ge 1$
and non-zero there. Then
$$
\sum_{n \le x} a_n \sim \frac{d(\alpha)x}{(\log x)^{\alpha}}
$$
where $d(\alpha) = h(1)/ \Gamma(1-\alpha)$.
\end{thm}

 In $\Re(s) > 1$, since $f$ does not vanish in this region, one may take logarithms
on either side of \eqref{f}. Now applying the series expansion for logarithms, we get
$$
\log f(s)
=  \sum_{(\p, \K_{\mathfrak{f}}/\K)=1 \atop (\p, \mathcal{C}\O_{\K})=\O_{\K}} \frac{1}{\N\p^s} + \sum_{(\p, \K_{\mathfrak{f}}/\K)=1 \atop (\p, \mathcal{C}\O_{\K})=\O_{\K}} \sum_{m\ge 2} \frac{(-1)^{m+1}}{m\N\p^{ms}}
= \sum_{(\p, \K_{\mathfrak{f}}/\K)=1, \atop (p, \mathcal{C}\O_{\K})=\O_{\K}} \frac{1}{\N\p^s} + \theta(s)
$$
where $$\theta(s) = \sum_{(\p, \K_{\mathfrak{f}}/\K)=1 \atop (\p, \mathcal{C}\O_{\K})=\O_{\K}} \sum_{m\ge 2} \frac{(-1)^{m+1}}{m\N\p^{ms}} .$$
Note that  for $\Re(s) = \sigma$,
$$
\left|\theta(s)  \right| \le \sum_\p \sum_{m\ge 2} \frac{1}{m\N\p^{m\sigma}} \le \sum_\p \frac{1}{\N\p^{\sigma} (\N\p^{\sigma}-1)} \le n_{\K} \sum_p  \frac{1}{p^{\sigma} (p^{\sigma}-1)}
$$
where $n_{\K} = [\K:\Q]$.
Therefore, $\theta(s)$ is holomorphic on $\Re(s) > 1/2$.
By orthogonality of generalised Dirichlet characters modulo $\mathfrak{f}$ we get
$$
\sum_{(\p, \K_{\mathfrak{f}}/\K)=1 \atop (\p, \mathcal{C}\O_{\K})=\O_{\K}} \frac{1}{\N\p^s} = \frac{1}{|H_{\mathfrak{f}}(\K)|}\sum_{(\p, \mathcal{C}\O_{\K})=\O_{\K}} \sum_{\chi \bmod \mathfrak{f}}\frac{\chi(\p)}{\N\p^s}.
$$
In $\Re(s) > 1$ we can interchange the sums to get
\begin{eqnarray*}
\sum_{(\p, \K_{\mathfrak{f}}/\K)=1 \atop (\p, \mathcal{C}\O_{\K})=\O_{\K}} \frac{1}{\N\p^s}
& = & \frac{1}{|H_{\mathfrak{f}}(\K)|} \sum_{\chi \bmod \mathfrak{f}} \sum_{ (\p, \mathcal{C}\O_{\K})=\O_{\K}} \frac{\chi(\p)}{\N\p^s}\\
& = &
 \frac{1}{|H_{\mathfrak{f}}(\K)|} \sum_{\chi \bmod \mathfrak{f}} \sum_{\p} \frac{\chi(\p)}{\N\p^s} +  \frac{1}{|H_{\mathfrak{f}}(\K)|} \sum_{\chi \bmod \mathfrak{f}} \sum_{\p \mid \mathcal{C}\O_{\K}}\frac{\chi(\p)}{\N\p^s}.
\end{eqnarray*}
However this now gives us
\begin{eqnarray}\label{psum}
\sum_{(\p, \K_{\mathfrak{f}}/\K)=1 \atop (\p, \mathcal{C}\O_{\K})=\O_{\K}} \frac{1}{\N\p^s}  =  \frac{1}{|H_{\mathfrak{f}}(\K)|} \sum_{\chi \bmod \mathfrak{f}} (\log L(s,\chi) + \theta_{1,\chi}(s))
\end{eqnarray}
where $\theta_{1, \chi}(s)$ is holomorphic on $\Re(s) \ge 1/2$ for all $\chi$ modulo $\mathfrak{f}$ (by the same argument
as seen above for $f(s)$).
Since $L(s,\chi_0)$ extends to a meromorphic function on $\C$ with only a simple pole at $s=1$
we have
$$
L(s, \chi_0) = \frac{\theta_2(s)}{(s-1)}
$$
with $\theta_2(s)$ entire. We have
$$
\theta_2(s) = (s-1)L(s, \chi_0) = (s-1)\zeta_{\K}(s) \prod_{\p \mid \mathfrak{f}} \left( 1- \frac{1}{\N\p^s} \right)
$$
and the product
$$
\prod_{\p \mid \mathfrak{f}} \left( 1- \frac{1}{\N\p^s} \right)
$$
is entire and non-zero in $\Re(s) \ge 1/2$.
Now by the zero free region for $\zeta_{\K}(s)$ (see lemma 8.1 and 8.2 of \cite{LO}),
we have a simply connected region containing $\Re(s) \ge 1$ in which $\theta_2(s)$
is non-zero. Therefore
$$
\log L(s, \chi_0) = \log \frac{1}{(s-1)} + \theta_3(s)
$$
where $\theta_3(s) = \log \theta_2(s)$ which is holomorphic in $\Re(s) \ge 1$.
For $\chi \neq \chi_0$,
$L(s, \chi)$ extends to an entire function (see corollary 8.6 on page 503 of \cite{JN}). Again by the zero free regions for each $L(s,\chi)$ (from the zero free region for $\zeta_{\K_{\mathfrak{f}}}(s)$ and the factorisation $\zeta_{\K_{\mathfrak{f}}}(s) = \prod_{\chi \bmod \mathfrak{f}} L(s,\chi^*)$)
we have a simply connected region containing $\Re(s) \ge 1$ in which $L(s, \chi)$ is non-zero.
In this region it follows that $\log L(s, \chi)$ can be defined and is holomorphic.
Combining all these observations  and substituting in \eqref{psum} we get
$$
\sum_{(\p, \K_{\mathfrak{f}}/\K)=1 \atop (\p, \mathcal{C}\O_{\K})=\O_{\K}} \frac{1}{\N\p^s}  =  -\frac{1}{|H_{\mathfrak{f}}(\K)|} \log(s-1) + \theta_4(s)
$$
where $\theta_4(s)$ is holomorphic on $\Re(s) \ge 1$.
Taking exponentials, we get
$$
f(s) = \frac{h(s)}{(s-1)^{\frac{1}{|H_{\mathfrak{f}}(\K)|}}}
$$
with $h(s) = e^{\theta(s)+\theta_4(s)}$ holomorphic in $\Re(s) \ge 1$ and non-zero there.
We can now apply the above Tauberian theorem to deduce that
$$
\#\{ \alpha\O_{\K} :\alpha\O_{\K}\in \W(X) \}  \sim  \frac{d_{\mathcal{C}}(1-1/|H_{\mathfrak{f}}(\K)|)X}{(\log X)^{1-\frac{1}{|H_{\mathfrak{f}}(\K)|}}}.
$$
Therefore, we have
$$
\mathcal{F}(X) \sim \frac{d_{\mathcal{C}}(1-1/|H_{\mathfrak{f}}(\K)|)X}{(\log X)^{1-\frac{1}{|H_{\mathfrak{f}}(\K)|}}}.
$$

\section{Lower bound for
$2^{{\rm rk}_4({\rm Cl}_{\L})}$.}\label{lowerbound}

Throughout this section we assume $\L \in \mathcal{F}$.
\begin{defn}
We say that a class $[\a]$ in the group $Cl_{\L}$ is strongly ambiguous if it contains
an ideal $\a$ such that $\a = \sigma(\a)$ for the generator $\sigma \in {\rm Gal}(\L/\K)$.
We denote the subgroup of all such classes in the class group by ${\rm Am}_{st}(\L/\K)$.
\end{defn}

We now recall the class number formula for strongly ambiguous classes.

\begin{thm}[{\rm see \cite{FL}}]
The number of strongly ambiguous classes is given by
$$
\#{\rm Am}_{st}(\L/\K) = h(\K) \cdot \frac{2^{t_{1,\L}+t_{2,\L}-1}}{[\O_{\K}^* : N_{\L/\K}(\O_{\L}^*)]}
$$
where $t_{1,\L}$ is the number of prime ideals in $\O_{\K}$ ramified in $\L$ and $t_{2,\L}$ is the number of  primes at infinity of $\K$ which ramify in $\L$.
\end{thm}

The subgroup of classes generated by the ramified primes in $\L/\K$ is contained in
the group of strongly ambiguous classes of the class group. On the other hand
suppose we consider a strongly ambiguous class represented by a fractional ideal
$\a$ of $\L$ such that $\a = \sigma(\a)$, then
let $\a = \p_1^{s_1} \cdots \p_k^{s_k}$ with all the $\p_i$
being distinct prime ideals of $\O_{\L}$. Since $\sigma(\a) = \a$,
for every prime $\p_i$, $\p_i = \sigma( \p_j)$ for some $1 \le j \le k$
and $s_i =s_j$ (by unique factorisation of ideals in a Dedekind domain and distinctness of the $\sigma(\p_j)$'s). This implies that if $\p_i$ lies above a prime of $\K$ that splits in $\L$,
$\p_i\p_j  = \p_i \sigma(\p_i)$ divides $\a$ and it  is principal. If $\p_i = \sigma(\p_i)$ it lies above a prime of $\K$ that is either inert or ramified.
If $\p_i$ is above an inert prime, it is already principal since $\O_{\K}$ is a PID. Any other $\p_i$
must lie above a ramified prime. This implies that ${\rm Am}_{st}(\L/\K)$ lies in the
subgroup generated by the classes of ramified primes of $\L/\K$.

\begin{rmk}
We note that in our context $h(\K)=1$ and we set $[\O_{\K}^* : N_{\L/\K}(\O_{\L}^*)] = 2^{\tilde{e}_{\L}}$.
\end{rmk}

\begin{defn}
We will use $\tilde{\mathfrak{B}}_{\L}$ to denote the set of ideals
$$
\{ \mathfrak{P}_1^{e_1}\mathfrak{P}_2^{e_2}\cdots \mathfrak{P}_{t_{1,\L}}^{e_{t_{1,\L}}} : \mathfrak{P}_i \mid 4\alpha\O_{\L} \text{ prime, } e_i \in \{0,1\} \text{ for all } i \in \{1,\ldots t_{1,\L}\} \}.
$$
Further let
${\mathfrak{B}}_{\L}$  denote the set of ideals
$$
\{ \mathfrak{P}_1^{e_1}\mathfrak{P}_2^{e_2}\cdots \mathfrak{P}_{t_{1,\L}-1}^{e_{t_{1,\L}-1}} : \mathfrak{P}_i \mid \alpha\O_{\L} \text{ prime, } e_i \in \{0,1\} \text{ for all } i \in \{1,\ldots t_{1,\L}-1\} \}.
$$
\end{defn}
\begin{rmk}
We have  ${\rm Am}_{st}(\L/\K) = \{ [\b] : \b \in \tilde{\mathfrak{B}}_{\L} \}$.
Given a $[\b] \in {\rm Am}_{st}(\L/\K)$ there is a bijection between
$$
\{ \b_1 \in \tilde{\mathfrak{B}}_{\L} : \b_1 \text{ principal }\} \leftrightarrow \{ \b_2 \in  \tilde{\mathfrak{B}}_{\L} :  \b_2 \in [\b]\}.
$$
The map from the left to right is obtained by sending  $\b_1$ to $\b\b_1/{\rm gcd}(\b, \b_1)^2$
Further, the inverse map is also given by sending $\b_2$ to $\b\b_2/{\rm gcd}(\b, \b_2)^2$.
Since
$$
\frac{\b \cdot \b\b_1/{\rm gcd}(\b, \b_1)^2}{{\rm gcd}(\b, \b\b_1/{\rm gcd}(\b, \b_1)^2)^2} = \b_1
$$
these maps constitute bijections.
We now conclude
that each class $[\b] \in \tilde{\mathfrak{B}}_{\L}$ has exactly $k$ representatives in $\tilde{\mathfrak{B}}_{\L}$ where
$$
k = \#  \{ \b_1 \in \tilde{\mathfrak{B}}_{\L} : \b_1 \text{ principal }\}.
$$
We now observe that
$$
\# \tilde{\mathfrak{B}}_{\L} = \sum_{[\b] \in {\rm Am}_{st}(\L/\K)} \# \{ \b_2 \in  \tilde{\mathfrak{B}}_{\L} :  \b_2 \in [\b]\} = k \cdot \#{\rm Am}_{st}(\L/\K).
$$
It follows from the ambiguous class number formula that
$$
k = 2^{1-t_{2,\L}}[\O_{\K}^* : N_{\L/\K}(\O_{\L}^*)] = 2^{\tilde{e}_\L+1 - t_{2,\L}}.
$$
Further, we have
$$
[\O_{\K}^* : N_{\L/\K}(\O_{\L}^*)] \le  [\O_{\K}^* :  (\O_{\K}^*)^2]  \le 2^{r_\K+1}.
$$
where $r_\K$ is the unit rank of $\O_{\K}^*$.
So $\tilde{e}_\L \le {r_\K+1}$.
\end{rmk}
\begin{lem}\label{8}
For any $\L = \L_{\alpha} \in \mathcal{F}$,
$$
2^{{\rm rk}_4({{\rm Cl}_{\L}})} \ge \frac{1}{2^{r_\K+2}} |\{ \mathfrak{b} \in \mathfrak{B}_{\L} : [\b]\in {\rm Cl}_{\L},  [\mathfrak{b}] = [\mathfrak{a}^2 ] \text{ for  some non-zero fractional ideal } \mathfrak{a} \text{ of } \O_{\L}  \}|.
$$
\end{lem}
\begin{proof}
By definition
$$
{\rm rk}_4( {\rm Cl}_{\L}) = {\rm dim}_{\mathbb{F}_2} \left({\rm Cl}_{\L}^2/{\rm Cl}_{\L}^4\right).
$$
Therefore
$$
2^{{\rm rk}_4( {\rm Cl}_{\L})} = | {\rm Cl}_{\L}^2/{\rm Cl}_{\L}^4 | = |\{B^2 \in {\rm Cl}_{\L} : B^4 = [\O_{\L}]\}|.
$$
However
$$
\{B^2 \in {\rm Cl}_{\L} : B^4 = [\O_{\L}]\} \supseteq \{ [\b]\in {\rm Cl}_{\L} : \mathfrak{b} \in \tilde{\mathfrak{B}}_{\L}, [\mathfrak{b}] = [\mathfrak{a}^2] \text{ for some non-zero fractional ideal } \mathfrak{a} \text{ of } \O_{\L}\}.
$$
Therefore we have
\begin{eqnarray*}
2^{{\rm rk}_4({{\rm Cl}_{\L}})} & \ge & \frac{1}{2^{r_\K+2}} |\{ \mathfrak{b} \in \tilde{\mathfrak{B}}_{\L} :  [\b]\in {\rm Cl}_{\L}, [\mathfrak{b}] = [\mathfrak{a}^2 ] \text{ for  some non-zero fractional ideal } \mathfrak{a} \text{ of } \O_{\L}  \}| \\
& \ge & \frac{1}{2^{r_\K+2}} |\{ \mathfrak{b} \in {\mathfrak{B}}_{\L} :  [\b]\in {\rm Cl}_{\L}, [\mathfrak{b}] = [\mathfrak{a}^2 ] \text{ for  some non-zero fractional ideal } \mathfrak{a} \text{ of } \O_{\L}  \}|.
 \end{eqnarray*}
\end{proof}
We now define an analogue of the usual Hilbert symbol and some of its properties
with an aim to characterise the lower bound in \lemref{8}.
\begin{defn}
For $a, b \in \K^*$, we define
\begin{eqnarray*}
(a | b) ~:~ \K^* \times \K^* \to \{0,1\}
\end{eqnarray*}
where
\begin{eqnarray*}
(a|b) = \begin{cases}
1 &\text{ if } x^2 = ay^2 + bz^2 \text{ has a non-zero solution in }\K^3\\
0 & \text{otherwise .}
\end{cases}
\end{eqnarray*}
\end{defn}

The next proposition shows the equivalence of three useful properties
which will be used in simplifying the bound in \lemref{8}.
\begin{prop}\label{equiv}
For any $\L = \L_{\alpha} \in \mathcal{F}$ and $\mathfrak{b} \in \mathfrak{B}_{\L}$, we denote by $b \O_{\K}$ the ideal $\N_{\L/\K}(\mathfrak{b})$.
Then the following are equivalent:
\begin{enumerate}
\item  $[\mathfrak{b}] = [\mathfrak{a}^2]$ for some non-zero fractional ideal $\mathfrak{a}$ of $\O_{\L}$.
\item  $b\varepsilon$ is a norm of an element in $\L^*$ for some $\varepsilon \in \O_{\K}^*$.
\item $(\alpha | b\varepsilon)=1.$
\end{enumerate}
\end{prop}

\begin{proof}
$(1) \iff (2) $:\\
For the forward implication, suppose $(\gamma_2) = \a^2 \b^{-1}$. Taking norms on both sides we get
$$
\N_{\L/\K}(\gamma_2\O_{\L}) = b\frac{\N_{\L/\K}(\a^2)}{b^2}
$$
where $b\O_{\K} = \N_{\L/\K}(\b)$. Since $\O_\K$ is a PID we have a $\gamma_3 \in \K^*$ such that $(\gamma_3) = \N_{\L/\K}(\a)$.
By Lemma 13 on page 25 of \cite{SwD}, we have
$$
\N_{\L/\K}(\gamma_2 \O_{\L}) =
N_{\L/\K} (\gamma_2) \O_{\K}.
$$
Therefore
$$
\N_{\L/\K}(\gamma_2 \O_{\L}) =
N_{\L/\K} (\gamma_2) \O_{\K} = b \frac{ \gamma_3^2}{b^2} \O_{\K}= b  N_{\L/\K}(\gamma_3/{b}) \O_{\K}.
$$
This gives us assertion (2).\\
\noindent
Conversely, let $b\varepsilon = N_{\L/\K}(\gamma_4)$ for some $\gamma_4 \in  {\L^*}$.
We know that $\gamma_4$ must be of the form $\gamma_5/\gamma_6$ where $\gamma_5 \in \O_{\L}$
and $\gamma_6 \in \O_{\K}$. Rationalising denominators, we get
$$
b\varepsilon \gamma_6^2 = N_{\L/\K}(\gamma_5).
$$
If a prime ideal $\p$ of $\O_{\K}$ divides $\gamma_5 \O_{\L}$, that is, $\p\O_{\L} \mid \gamma_5 \O_{\L}$,
then $\p^2 \mid b\varepsilon \gamma_6^2 \O_{\K}$. Since $b\O_{\K}$ is square-free, $\p^2 \mid \gamma_6^2 \O_{\K}$.
Therefore by dividing on both sides by a generator of the principal ideal $\p^2$, we may assume that $\p\O_{\L} \nmid \gamma_5 \O_{\L}$
for any prime ideal $\p$ of $\O_{\K}$.
However for every prime ideal $\p$ of $\O_{\K}$ dividing $\gamma_6\O_{\K}$ we must have $\p^2 \mid \N_{\L/\K}(\gamma_5\O_{\L})$.
Therefore there exists an ideal $I \subset \O_{\L}$ with $I^2 \mid \gamma_5 \O_{\L}$ and  $\N(I) = \p$.
Finally, there is a unique ideal of norm $b\O_{\K}$ given by some $\b \in \B_{\L}$. Combining the above
we get
$$
\gamma_5 \O_{\L} = \b\a^2 \text{ for some non-zero integral ideal } \a \text{ of } \O_{\L}.
$$
$(2) \iff (3)$:\\
For the forward implication, if $b\varepsilon$ is a norm in $\K$ of an element of $\L$, then there exist $ x,y \in \K$ such that  $x^2 - \alpha y^2 = b\varepsilon$. Therefore $(\alpha | b\varepsilon)=1$. Conversely, if $(\alpha | b\varepsilon)=1$, there exists a non-trivial tuple $(x,y,z) \in \K^3$ such that
$$
\alpha x^2 + b\varepsilon y^2 = z^2.
$$
Rewriting the same, we get
$$
 b\epsilon y^2 = z^2 - \alpha x^2.
$$
Since the ideal $\alpha\O_{\K}$ is square-free, $\sqrt{\alpha} \notin \K$. Therefore $y$ is non-zero.
Now dividing by $y^2$ we get the lemma.
\end{proof}

We will now prove some properties of the aforementioned analogue of the Hilbert symbol.
\begin{lem}\cite{FK07}\label{hprop}
We have $(a|b) = (b|a) = (a| -ab)$
and $(ac^2 |b) = (a|b)$.
\end{lem}
\begin{rmk}
For an arbitrary squarefree integral ideal $I \subset \O_{\K}$, we say that an element $\gamma_3 \in \O_{\K}$, with $(\gamma_3\O_{\K}, I) = \O_{\K}$,
is a square modulo $I$ if it is a quadratic residue modulo every prime ideal $\p \mid I$.
\end{rmk}

\begin{lem}\label{splitting1}
For any ideal $\alpha\O_{\K} \in \W$,
we have
$$
2 \mid [(\O_{\K}/\p)^*:\O_{\K}^* \bmod \p], \text{ for all } \p \mid \alpha\O_{\K}.
$$
\end{lem}
\begin{proof}
By definition of $\alpha\O_{\K}$, $(\p, \K_{\mathfrak{f}}/\K) = 1$ for all $\p \mid \alpha\O_{\K}$ .
 Now consider the field $\K(\sqrt{\varepsilon})$ for any element of the set $S$ considered in definition \ref{oneS}, say $\varepsilon$.
 Since $\K(\sqrt{\varepsilon}) \subset \K_{\mathfrak{f}}$,
 $\p$ splits in $\K_1=\K(\sqrt{\varepsilon})$. To apply the Dedekind-Kummer theorem (see page 47 of \cite{JN}) , we
 note that any prime $\p \mid \alpha\O_{\K}$  is co-prime to the conductor of $\O_{\K}[\sqrt{\varepsilon}]$
 in $\O_{\K_1}$. Therefore $X^2 -\varepsilon$ splits modulo $\p$. Therefore for any element, $\varepsilon$ of $S$, we have
$$
2 \mid  [(\O_{\K}/\p)^*: \langle \varepsilon \rangle \bmod \p].
$$
Applying the same to all the elements of $S$ and using the fact that $(\O_\K/\p)^*$ is cyclic,
we have the lemma.
\end{proof}
\begin{lem}\label{laststep}
Let $\alpha\O_{\K} \in \W$.
Suppose $\alpha\O_{\K} = ab\O_{\K}$ and $(\alpha | b\varepsilon)=1$ for some $\varepsilon \in \O_{\K}^*$, then
$a$ is a square modulo $b\O_{\K}$ and $b$ is a square modulo $a \O_{\K}$.
\end{lem}
\begin{proof}
Since $\alpha\O_{\K} = ab\O_{\K}$, we have $ab \varepsilon_1 =\alpha $ for some $\varepsilon_1 \in \O_{\K}^*$.
From Lemma \ref{hprop}, we have
$$
1 = (\alpha| b\varepsilon) = ( ab \varepsilon_1 | b \varepsilon) = ( -ab \varepsilon_1 \cdot b \varepsilon | b \varepsilon) = (a \varepsilon_3 | b\varepsilon)
$$
where each $ \varepsilon_i \in \O_{\K}^*$.
Since we have $(a\varepsilon_3 | b\varepsilon)=1$, we have
\begin{eqnarray}\label{eqn}
b \varepsilon z^2 = x^2 - a\varepsilon_3 y^2  \text{ for some non-zero tuple } (x,y,z) \in \K^3.
\end{eqnarray}
Since the ideal $a\O_{\K}$ is square-free, $\sqrt{a\varepsilon_3} \notin \K$. Therefore $z$ is non-zero.
Further $z$ is coprime to every prime ideal dividing $a\O_{\K}$. This is because if $\p \mid {\rm gcd}(z\O_{\K}, a\O_{\K})$
then $\p^2 \mid ay^2\O_{\K}$. Therefore an odd power of $\p$ will divide $ay^2\O_{\K}$ ($a\O_{\K}$ is square free) and an even power of $\p$ will divide the other two terms (since $\gcd(a\O_{\K}, b\O_{\K}) = \O_{\K}$).
Similarly $\sqrt{b\varepsilon} \notin \K$ and $y$ is coprime to every prime ideal dividing $b\O_{\K}$.
Further, by \lemref{splitting1}
$$
2 \mid [(\O_{\K}/\p)^*:\O_{\K}^* \bmod \p], \text{ for all } \p \mid \alpha\O_{\K}
$$
and in particular for all $\p \mid a\O_{\K}$, we get that $\O_{\K}^* \bmod \p$ is in the
subgroup of quadratic residues modulo $\p$.
Now reading \eqref{eqn} modulo $a\O_{\K}$, we get
that $b$ is a square modulo $a\O_{\K}$.
The argument for $a$ modulo $b\O_{\K}$ is similar.
\end{proof}

\begin{thm}\label{algresult}
We have, for $\alpha\O_{\K} \in \W$ with $\L = \L_{\alpha} \in \mathcal{F}$,
\begin{eqnarray*}
2^{{\rm rk}_4({{\rm Cl}_{\L}})} & \ge & \frac{1}{2^{r_\K+2}} |\{ (a\O_{\K}, b\O_{\K}) : a\O_{\K}, b\O_{\K} \text{ square free }, \alpha \O_{\K} = ab\O_{\K},  \\
&& ~~
a \text{ is a square modulo } b\O_{\K} \text{ and } b \text{ is a square modulo } a \O_{\K}\}|
\end{eqnarray*}
\end{thm}
\begin{proof}
\begin{eqnarray*}
2^{{\rm rk}_4({{\rm Cl}_{\L}})} &\ge &  \frac{1}{2^{r_\K+2}} |\{ \mathfrak{b} \in \mathfrak{B}_{\L} :  [\b]\in {\rm Cl}_{\L}, [\mathfrak{b}] = [\mathfrak{a}^2 ] \text{ for  some non-zero fractional ideal } \mathfrak{a} \text{ of } \O_{\L}  \}|\\
& = &   \frac{1}{2^{r_\K+2}} |\{ \mathfrak{b} \in \mathfrak{B}_{\L} :  \N_{\L/\K}(\b)= b\O_{\K}, (\alpha|b\varepsilon)=1 \text{ for some } \varepsilon \in \O_{\K}^* \}|
 \end{eqnarray*}
 Since there is a unique ideal $\b \in \B_{\L}$ of norm $b\O_{\K}$ for every $b\O_{\K} \mid \alpha \O_{\K}$, we get
\begin{eqnarray*}
2^{rk_4({{\rm Cl}_{\L}})}
& \ge &   \frac{1}{2^{r_\K+2}} |\{ b\O_{\K} \mid \alpha \O_{\K} : (\alpha|b\varepsilon)=1 \text{ for some } \varepsilon \in \O_{\K}^* \}|\\
& \ge &  \frac{1}{2^{r_\K+2}} |\{ (a\O_{\K}, b\O_{\K}) : \alpha \O_{\K} = ab\O_{\K}, (\alpha|b\varepsilon)=1 \text{ for some } \varepsilon \in \O_{\K}^* \}|
 \end{eqnarray*}
 Now by \lemref{laststep}, we have the theorem.
\end{proof}
\section{Algebraic preliminaries}\label{algprelims}
\begin{defn} Let $\p$ be an any prime ideal in $\O_{\K}$ and let $a \in \O_{\K}$. We define the quadratic residue symbol as
\[
\left(\frac{a}{\p}\right)=\begin{cases}
0\, \quad \textit{if}\, a\in \p,\\
1\,\quad \textit{if}\, a \not \in \p \text{ and } a\, \textit{is square } \bmod \p,\\
-1 \quad\textit{otherwise}.
\end{cases}
\]
\end{defn}
Let $\p$ be a prime ideal and $I$ be an ideal in $\O_{\K}$ such that $\p$ and $I$ are co-prime. Since $\O_{\K}$ is a PID, we can express $I$ as $I=i\O_{\K}$. For $\p \mid \alpha\O_{\K}$, we define $\Phi_{\p}(I)$ as:
\[
\Phi_{\p}(I):=\left(\frac{I}{\p}\right):=\left(\frac{i}{\p}\right)
\]
Since $2|\left[\left(\O_{\K}/\p\right)^{*}:\O^{*}_{\K}\pmod {\p}\right]$, the mapping  is well-defined. Firstly, we will define ray class group mod $\p$ of $\K$ and then we will prove that $\Phi_{\p}$ is a character on that group.
We generically denote places by the symbol $\nu$, but for non-archimedean places, we may use $\q$ to denote both a prime of $\K$ and the place corresponding to the absolute value $|.|_{\q}$. We write $\nu|\infty$ to indicate that $\nu$ is an archimedean
place, which is real or complex and $M_{\K}$ to be the set of all inequivalent places of $\K$.
\begin{defn}
Let $g:M_{\K} \rightarrow \mathbb{Z}_{\ge 0}$ be a function  with finite support such that for $\nu\in M_{\K}$ and $\nu|\infty$ we have $g(\nu) \le 1$ with $g(\nu) = 0 $ unless $\nu$ is a real place. Then any modulus $\b$ in $\K$ can be viewed as a formal product
\[
\b=\b_0\b_{\infty}, \quad \textit{with} \, \b_0=\prod_{\q\nmid\infty, \q \mid \b}\q^{g(\q)} \quad
\textit{and}\, \b_{\infty}=\prod_{\nu|\infty}\nu^{g(\nu)}.
\]
where $\b_0$ corresponds to an $\O_{\K}$ -ideal and $\b_{\infty}$ represents a subset of the real places of $\K$.
\end{defn}
Now we define the following notation in $\O_{\K}$:
\begin{enumerate}
\item $\mathcal{I_{\K}}$ be the set of all non-zero fractional ideals in $\O_{\K}$.
\item $\mathcal{I^{\b}_{\K}}\subseteq\mathcal{I_{\K}}$ is the subgroup of fractional
ideals which is prime to $\b$.
\item $\K^{\b}\subseteq \K^{*}$ is the subgroup of elements $\alpha \in \K^{*}$ for which $(\alpha) \in \mathcal{I^{\b}_{\K}}.$
\item $\K^{\b,1}\subseteq \K^{\b}$ is the subgroup of elements $\alpha \in \K^{\b}$ for which $\nu_{\q}(\alpha - 1) \ge \nu_{\q}(\b_0)$ for all primes $\q|\b_0$
and $\alpha_{\nu} > 0$ for $\nu|\b_{\infty}$ (here $\alpha_{\nu} \in \mathbb{R}$ is the image of $\alpha$ under the real-embedding $\nu$).
\item $\mathcal{P}^{\b}_{\K}\subseteq \mathcal{I^{\b}_{\K}}$  is the subgroup of principal fractional ideals $(\alpha) \in \mathcal{I^{\b}_{\K}}$ with $\alpha \in \K^{\b,1}$.
\end{enumerate}
\begin{defn}
The ray class group of $\K$ for the modulus $\b$ is the quotient
\[
H_{\b}(\K):=\mathcal{I}_{\K}^{\b}/\mathcal{P}^{\b}_{\K}.
\]
\end{defn}
Recall that for $\p \mid \alpha\O_{\K} \in \W$, $\Phi_{\p}(I)=\left(\frac{i}{\p}\right)$, if $I=i\O_{\K}$ with
${\rm gcd}(I, \p) =\O_\K$.
Since the power residue symbol is multiplicative, it immediately follows that $\Phi_{\p}$ is also multiplicative.
If $i\varepsilon \equiv 1 \bmod \p$ for some $\varepsilon \in \O_{\K}^*$ then $\Phi_{\p}(I)=1$.
Therefore, $\Phi_{\p}$ is a character on $H_{\p}(\K)$.
 Such a character is called a generalized Dirichlet character (a special instance of a Hecke character).
Let $\a$ be a square-free  ideal in $\O_{\K}$ such that $\a \mid \alpha \O_{\K} \in \W$. Then we define the symbol $\left(\frac{.}{\a}\right)$ from $ H_{\a}(\K)$ to $\{\pm 1\}$ by
\[
\left(\frac{.}{\a}\right)=\prod_{\q|\a}\left(\frac{.}{\q}\right)
\]
This map is multiplicative since the power residue symbol is multiplicative.
Further, for any ideal $I = i\O_{\K}$ with ${\rm gcd}(I, \a) = \O_{\K}$ and $i\varepsilon  \equiv 1 \bmod \a$ for some $\varepsilon \in \O_{\K}^*$, we have
\[
\left(\frac{i}{\a}\right)=\prod_{\p|\a}\left(\frac{i}{\p}\right) = 1.
\]

This makes  $\left(\frac{.}{\a}\right)$ a character on $ H_{\a}({\K})$.
We now recall a theorem of Bauer from Class field theory.
\begin{thm}{\rm (Bauer, Theorem 8.19 of \cite{Cox})}\label{Bauer}
Given two finite degree Galois extensions $\mathbf{F}_1$ and $\mathbf{F}_2$ of a number field $\K$, if
$$
\left|\{ \p \subset \O_{\K} : \p \text{ splits in } \mathbf{F}_1  \} \setminus \{ \p \subset \O_{\K} : \p \text{ splits in } \mathbf{F}_2  \} \right| < \infty,
$$
then $\mathbf{F}_2 \subset \mathbf{F}_1.$
\end{thm}
We now introduce some notation. Let
$s \O_{\K}$ be an ideal of $\O_{\K}$.
If $s$ satisfies \eqref{cond2}, we set $\varepsilon_s = 1$.
If $s$ does not satisfy \eqref{cond2}, we choose $\varepsilon \in \O_{\K}^*$
such the order of $\varepsilon$ modulo each $\p^2$ for $\p \mid 2\O_{\K}$ is
$2$ (This is  condition \eqref{units} on $\K$) . We now set $\varepsilon_s =\varepsilon$.
We now consider the field $\L_{s\varepsilon_s}/\K$.
and apply the above theorem of Bauer to prove the following lemma.
\begin{lem}\label{primitivity rmk}
Consider the map
$$
\left(\frac{s\varepsilon_s}{\cdot}\right): \mathcal{I}^{s\O_{\K}}_{\K} \rightarrow \{\pm 1\},
$$
given by
$$
\bigg(\frac{s\varepsilon_s}{\partial}\bigg)= \prod_{\mathfrak{p}^r \mid\mid \partial}\bigg(\frac{s\varepsilon_s}{\mathfrak{p}}\bigg)^r, \ \text{for all} \ (\partial, s\O_{\K})=1.
$$
The map, $\big(\frac{s\varepsilon_s}{.}\big)$, defines a primitive character of $H_{\d_{\L_{s\varepsilon_s}/\K}}(\K)$.
\end{lem}
\begin{proof}
Clearly, $\big(\frac{s\varepsilon_s}{.}\big)$ is multiplicative. The extension $\L_{s\varepsilon_s}/\K$ has relative degree 2 and relative discriminant $\d_{\L_{s\varepsilon_s}/\K}=4s\mathcal{O}_{\K}$. We may think of $\big(\frac{s\varepsilon_s}{\cdot}\big)$ as a character on $\mathcal{I}_{\K}^{\d_{\L_{s\varepsilon_s}/\K}}$
defined by
$$
\bigg(\frac{s\varepsilon_s}{\partial}\bigg)= \prod_{\mathfrak{p}^r \mid\mid \partial}\bigg(\frac{s\varepsilon_s}{\mathfrak{p}}\bigg)^r,  \text{for all} ~{\rm gcd}(\partial, \d_{\L_{s\varepsilon_s}/\K})=1.
$$
By \lemref{mono} and Dedekind–Kummer theorem, we know that for any prime ideal $\mathfrak{q}$ satisfying the condition ${\rm gcd}(\mathfrak{q}, 4s\O_{\K})=\mathcal{O}_{\K}$,
\begin{eqnarray}\label{splitting}
\bigg(\frac{s\varepsilon_s}{\mathfrak{q}}\bigg)=1 & \text{ if and only if } &   X^2 -s\varepsilon_s \text{ splits modulo } \q \nonumber\\
& \text{ if and only if } &(\mathfrak{q}, \L_{s\varepsilon_s}/\K)=1.
\end{eqnarray}
 Here $(\mathfrak{q}, \L_{s\varepsilon_s}/\K)$ denotes the Artin symbol of $\q$ with respect to the relative extension $ \L_{s\varepsilon_s}/\K$. Hence for an ideal $\partial$, ${\rm gcd}(\partial, 4s\O_{\K})=\mathcal{O}_{\K}$,
\begin{equation}\label{trivial}
\bigg(\frac{s\varepsilon_s}{\partial}\bigg)=1  \text{ if and only if } \prod_{\mathfrak{q}^r  \mid\mid \partial}(\mathfrak{q}, \L_{s\varepsilon_s}/\K)^r=1.
\end{equation}

By Conductor-discriminant formula, we have that the conductor of the extension $\L_{s\varepsilon_s}/\K$ satisfies $\mathfrak{f}_{\L_{s\varepsilon_s}/\K}= \d_{\L_{s\varepsilon_s}/\K}$. This implies $\L_{s\varepsilon_s} \subseteq \K_{\d_{\L_{s\varepsilon_s}/\K}}$, where $\K_{\d_{\L_{s\varepsilon_s}/\K}}$
is a ray class field of modulus $\d_{\L_{s\varepsilon_s}/\K}$ with respect to the number field $\K$.

By Class field theory, $  H_{\d_{\L_{s\varepsilon_s}/\K}}(\K)  \cong $ Gal$(\K_{\d_{\L_{s\varepsilon_s}/\K}}/\K)$ via the map
$$ [\a] \rightarrow \prod_{\mathfrak{q}^{r} \Vert \a} (\mathfrak{q}, \K_{\d_{\L_{s\varepsilon_s}/\K}}/\K)^{r}.$$
Now suppose we have an ideal $\partial=b\O_{\K}$ with $b \equiv 1 (\d_{\L_{s\varepsilon_s}/\K})$, then $[\partial]$ is trivial on the left. Therefore $\prod_{\mathfrak{q}^r \mid \mid \partial} (\mathfrak{q}, \K_{\d_{L_{s\varepsilon_s}/\K}}/\K)^r=1$. By properties of the Artin Symbol, we have
$$1=\prod_{\mathfrak{q}^r \mid \mid \partial} (\mathfrak{q}, \K_{\d_{\L_{s\varepsilon_s}/\K}}/\K)^r\big|_{\L_{s\varepsilon_s}}= \prod_{\mathfrak{q}^r \mid \mid \partial}(\mathfrak{q}, \L_{s\varepsilon_s}/\K)^r.$$
By \eqref{trivial} this implies that $\big(\frac{s\varepsilon_s}{\cdot}\big)$ is a character on $H_{\d_{\L_{s\varepsilon_s}/\K}}(\K)$.

We will now prove the claim about the primitivity of  $\big(\frac{s\varepsilon_s}{\cdot}\big)$ as a character on $H_{\d_{\L_{s\varepsilon_s}/\K}}(\K)$.
Suppose that $\big(\frac{s\varepsilon_s}{\cdot}\big)$ is a character on $H_{\a}(\K)$ for some $\a \mid \d_{\L_{s\varepsilon_s}/\K}$.
We have $ H_{\a}(\K) \cong$ Gal$(\K_{\a}/\K) $ via the map
$$\ [\b] \rightarrow\prod_{\mathfrak{q}^{r} \Vert \b} (\mathfrak{q}, \K_{\a}/\K)^{r}.$$
Therefore,
any prime $\q$ of $\O_\K$ which splits in $\K_{\a}/\K$ would also satisfy $[\q]=1$ in $H_{\a}(\K)$.
Since $\big(\frac{s\varepsilon_s}{\cdot}\big)$ is a character on $H_{\a}(\K)$, we would then have
$\big(\frac{s\varepsilon_s}{\q}\big)=1$.
By \eqref{splitting}, this means that if $(\q, 4s\O_{\K}) = \O_{\K}$, $\q$ will also split in $ \L_{s\varepsilon_s}/\K$.
Bauer's theorem (\thmref{Bauer}) would now imply that $ \L_{s\varepsilon_s} \subset \K_\a$, implying that $\a = \d_{\L_{s\varepsilon_s}/\K}$.
Therefore  $\big(\frac{s\varepsilon_s}{.}\big)$
 is a primitive character of $H_{\d_{\L_{s\varepsilon_s}/\K}}(\K)$.
  \end{proof}
\subsection{Analogue of quadratic reciprocity.}\label{reciprocity}

\begin{lem}\label{rayclassgp}
For an integral ideal $\a$ of $\O_{\K}$ we have
\[
H_{\a}(\K)\cong \left(\O_{\K}/\a\right)^*/\O^*_{\K} \bmod \a.
\]
Here $\O^*_{\K} \bmod \a = \{ u \bmod \a : u \in \O_{\K}^*\}$.
We shall denote this isomorphism by $\xi_{\a}$.
\end{lem}
\begin{proof}
Since $\O_{\K}$ is a PID, we may define
\begin{eqnarray*}
\xi :  \mathcal{I}_{\K}^{\a} & \rightarrow & \left(\O_{\K}/\a\right)^*/\O^*_{\K}\bmod \a\\
\b &\rightarrow& \bar{b},~~~~~~~~ \text{ where } \b=b\O_{\K}.
\end{eqnarray*}
Surjectivity of $\xi$ is obvious. We now consider injectivity.
To do so, we note
\begin{align*}
\ker \Psi = \{\b\in \mathcal{I}_{\K}^{\a}: \bar{b\varepsilon}\equiv 1 \bmod \a\, \textit{for some }\varepsilon\in \O^*_{\K}\}=\{\b\in \mathcal{I}_{\K}^{\a}:\b \, \textit{has a generator which is } 1 \bmod \,\a\}.
\end{align*}
Hence, $\mathcal{I}_{\K}^{\a}/{\mathcal{P}_{\K}^{\a}}\cong \left(\O_{\K}/\a\right)^*/\O^*_{\K}{(\a)}.$
\end{proof}

Class field theory tells us that
$$
\mathcal{I}_{\K}^{\a}/{\mathcal{P}_{\K}^{\a}} \cong \gal\left(\K_{\a}/\K\right)
$$
via the Artin map, denoted $\Psi_{\a}$.
From \lemref{rayclassgp}, we now have the following corollary.
\begin{cor}\label{isom}
For an integral ideal $\a$ of $\O_{\K}$ we have
\[
\gal\left(\K_{\a}/\K\right)\cong \left(\O_{\K}/\a\right)^*/\O^*_{\K} \bmod \a
\]
under the map $\xi_{\a} \circ \Psi_{\a}^{-1}$.
\end{cor}
Let $\p$ and $\q$  be two prime ideals in $\W$.
 Since $\O_{\K}$ is a PID, we have $\p=(p)$ and $\q=(q)$ for elements $p,q \in \O_{\K}$.
Without loss of generality, by the definition of $\W$, we may assume
$$q\equiv 1\bmod 4\O_{\K}.$$
Now,
\begin{eqnarray*}
\left(\dfrac{\q}{\p}\right)=\left(\frac{q}{\p} \right) =1 \text{ if and only if } x^{2} - q \text{ splits modulo } \p.
\end{eqnarray*}

\vspace{1mm}
\noindent
Let $\L^\prime_{\q}=\K(\sqrt{q})$, where $\dfrac{1+\sqrt{q}}{2}\in
\O_{\L^\prime_\q}.$ Then $\left\{1,\dfrac{1+\sqrt{q}}{2}\right\}$ is an basis of
$\L_{\q}^\prime$ over $\K$. Therefore $\d_{\L_{\q}^{\prime}/\K} \mid \q$. However
since $\q$ is prime and it ramifies in $\L_{\q}^\prime$, we have $\d_{\L_{\q}^{\prime}/\K} = \q$.
We claim now that $\O_{\L_{\q}^\prime} = \O_{\K}\left[\dfrac{1+\sqrt{q}}{2}\right]$.
\begin{lem}\label{monogenic}
Let $\K_1$ be any quadratic extension of $\K$, and $\O_{\K}$ be PID, then $\O_{\K_1}=\O_{\K}[\alpha]$ for some $\alpha\in \K_1$.
\end{lem}
By \lemref{monogenic}, we have $\O_{\L_{\q}^\prime} = \O_{\K}\left[\theta\right]$ for
some $\theta \in \O_{\L_{\q}^\prime}$. Since  $\dfrac{1+\sqrt{q}}{2}\in
\O_{\L^\prime_\q}$ and $\O_{\L^\prime_\q}$ is a free $\O_{\K}$ module with basis $\{1, \theta \}$, we have a matrix $M$ with entries in $\O_{\K}$
such that
$$
M \begin{pmatrix} 1 & 1 \\ \theta & \sigma(\theta) \end{pmatrix} =  \begin{pmatrix} 1 & 1 \\  \frac{1 +\sqrt{q}}{2} & \frac{1 -\sqrt{q}}{2} \end{pmatrix}.
$$
Since
$$
 {\rm det}(M)^2 {\rm disc}\{1 ,\theta\} \O_{\K} =  {\rm disc}\left\{1 ,  \frac{1 +\sqrt{q}}{2}\right\} \O_{\K} = \q =\d_{\L^{\prime}_{\q}/\K}
$$
we have that $ {\rm det}(M)^2 \in \O_{\K}^*$. But $ {\rm det}(M) \in \O_{\K}$ and therefore $ {\rm det}(M) \in \O_{\K}^*$.
This implies that $M$ is invertible and therefore we have our claim.
Now we have the following diagram
\[
\begin{tikzcd}
\K_{\q} \arrow[dash]{d}\arrow[bend right=47, dash]{d}{G_1} \arrow[dash]{d}\arrow[bend left=47, dash]{dd}{G}\\
\L^\prime_{\q} \arrow[dash]{d}{2}\\
\K
\end{tikzcd}
\]
Observe that, $G \cong \left(\O_{\K}/\q\right)^{*}/\O_{\K}^{*} \bmod \q$ (under the map $\xi_{\q} \circ \Psi_{\q}^{-1}$)  is cyclic.
Hence
\begin{eqnarray*}
\left(\dfrac{\q}{\p}\right)=1 &
\text{  if and only if } & x^2-q \text{ splits in modulo } \p, \\
& \text{  if and only if } & (2x-1)^2-q \text{ splits in modulo } \p ~(\p \text{ does not lie above } 2\Z), \\
& \text{  if and only if }&  \p \text{ splits in } \L_{\q}^{\prime} \text{ (by Dedekind-Kummer Theorem). }
\end{eqnarray*}
Now by the properties of the Artin Symbol, we have
\begin{eqnarray*}
\p \text{ splits in } \L_{\q}^{\prime}
&\text{  if and only if } &
(\p, \L_{\q}^\prime/\K) = 1\\
&\text{  if and only if } &
(\p, \K_{\q}/\K) \in G_1
\end{eqnarray*}
Let us now consider $\xi_{\q} \circ \Psi_{\q}^{-1} ((\p, \K_{\q}/\K)).$ By the definition of
the Artin map $\Psi_{\q}$, we have
$$
\xi_{\q} \circ \Psi_{\q}^{-1} ((\p, \K_{\q}/\K)) = \xi_{\q}([\p]).
$$
If the ideal $\p = p\O_{\K}$, we have $\xi_{\q}([\p]) = \bar{p} \in  \left(\O_{\K}/\q\right)^{*}/\O_{\K}^{*} \bmod \q$.
\begin{eqnarray*}
\p \text{ splits in } \L_{\q}^{\prime}
&\text{  if and only if } &
(\p, \K_{\q}/\K) \in G_1,\\
&\text{  if and only if } &
\bar{p} \text{ is in the unique subgroup of index } 2 \text{ in }  \left(\O_{\K}/\q\right)^{*}/\O_{\K}^{*} \bmod \q.
\end{eqnarray*}
The last observation follows from \corref{isom}.
We have a natural surjective homomorphism given by
\begin{eqnarray*}
\pi:  \left(\O_{\K}/\q\right)^{*} & \to &  \left(\O_{\K}/\q\right)^{*}/\O_{\K}^{*} \bmod \q\\
a \bmod \q & \to & \bar{a}.
\end{eqnarray*}
For $\p \in \W$, by \lemref{splitting1} we know that $ \O_{\K}^{*} \bmod \q$ is contained in the
subgroup of quadratic residues in  $\left(\O_{\K}/\q\right)^{*}$. Let us denote the subgroup
of quadratic residues in $ \left(\O_{\K}/\q\right)^{*}$ by $R_{\q}$.
Then we observe that $\pi(R_{\q})$ has index $2$ in  $\left(\O_{\K}/\q\right)^{*}$.
This is because
$$
\frac{ \left(\O_{\K}/\q\right)^{*}/\O_{\K}^{*} \bmod \q }{R_{\q} /\O_{\K}^{*} \bmod \q} \cong  \left(\O_{\K}/\q\right)^{*}/R_{\q}.
$$
Therefore, the unique subgroup of index $2$ in $\left(\O_{\K}/\q\right)^{*}/\O_{\K}^{*} \bmod \q$ is $\pi(R_{\q})$.
We may now conclude that
\begin{eqnarray*}
\p \text{ splits in } \L_{\q}^{\prime}
&\text{  if and only if } &
\bar{p} \text{ is in the unique subgroup of index } 2 \text{ in }  \left(\O_{\K}/\q\right)^{*}/\O_{\K}^{*} \bmod \q,\\
&\text{  if and only if } &
p \text{ is a quadratic residue modulo } \q,\\
&\text{  if and only if } & x^2-p \text{ splits in modulo } \q,\\
 &\text{  if and only if } & \left(\dfrac{\p}{\q}\right)=1.
\end{eqnarray*}

\vspace{1mm}
\noindent
Therefore, for primes $\p$ and $\q$ in $\W$, we have
$\left(\dfrac{\q}{\p}\right)=\left(\dfrac{\p}{\q}\right)$.

\vspace{1mm}
\noindent
By multiplicativity, we obtain the following lemma.
\begin{lem}\label{quadreciprocity}
Let $\partial_{\bar{v}}$ and $\partial_{\bar{u}}$ be two ideals in $\W$.Then
$\left(\dfrac{\partial_{\bar{v}}}{\partial_{\bar{u}}}\right)=\left(\dfrac{\partial_{\bar{u}}}{\partial_{\bar{v}}}\right)$.
\end{lem}
\section{Analytic preliminaries}\label{anaprelims}
\subsection{Divisor function and some variants}\label{anaprelims1}
We begin with an upper bound on the number of  squarefree integral ideals with norm atmost $x$ and a prescribed number of prime divisors.
\begin{lem}\label{l prime fact est}
There exists a constant $B_{0}=B_{0}(\K)$ such that for every $X \geq 3$ and $\ell \geq 1$ we have
$$\# \{ \mathfrak{a} \subseteq \O_{\K} : \mathfrak{N}(\mathfrak{a}) \leq X, \omega_{\K}(\mathfrak{a})= \ell, \mu^{2}(\mathfrak{a})=1 \} \ll_{\K} \frac{X}{\log X} \frac{(\log\log X+ B_{0})^{\ell-1}}{(\ell-1)!}.$$
Here $\N$ denotes the norm map from $\K$ to $\Q$.
\end{lem}
\begin{proof}
We prove this by induction. For $\ell=1$, the required inequality holds by prime ideal theorem.
We now assume the result for $\ell$ and prove for $\ell+1$. Let
$$
M_{\ell} (X)=\{ \mathfrak{a} \subseteq \O_{\K} : \mathfrak{N}(\mathfrak{a}) \leq X,  \omega_{\K}(\mathfrak{a})= \ell, \mu^{2}(\mathfrak{a})=1 \}.
$$
Let us consider an element $\a_1 = \p_1\p_2\cdots\p_{\ell+1} \in M_{\ell+1}(X)$ such that
$$\N(\p_1) < \N(\p_2) <\cdots< \N(\p_{\ell+1}).$$
Since $\N(\a_1) \le X$, $\N(\p_i) \le \sqrt{X}$ for all $1 \le i \le \ell$.
In other words $\N(\p_i^2) \le X$ for all $1 \le i \le \ell$. We also have
$$
\p_1\p_2\cdots\p_{i-1} \p_{i+1} \cdots \p_{\ell+1} \in M_{\ell}(X/\mathfrak{N}(\mathfrak{p}_i))~\text{ for all }~1 \le i \le \ell.
$$
This implies that
$$ \ell M_{\ell+1} (X) \le \sum_{\mathfrak{N}(\mathfrak{p}^{2}) \leq X}M_{\ell}(X/\mathfrak{N}(\mathfrak{p})).$$
Applying the induction hypothesis, we now have
\begin{align*}
\ell M_{\ell+1}(X)
&\ll \sum_{\mathfrak{N}(\mathfrak{p}^{2}) \leq X} \frac{X}{\mathfrak{N}(\mathfrak{p})} \cdot \frac{1}{\log (X/\mathfrak{N}(\mathfrak{p}))} \frac{(\log\log X+B_{1})^{\ell-1}}{(\ell-1)!}.
\end{align*}
Therefore, we have
\begin{align*}
M_{\ell+1} &\ll \frac{(\log\log X+B_{1})^{\ell-1}}{\ell!} X \sum_{\mathfrak{N}(\mathfrak{p}^{2}) \leq X} \frac{1}{\mathfrak{N}(\mathfrak{p}) \log (X/\mathfrak{N}(\mathfrak{p}))}.
\end{align*}
Since $\mathfrak{N}(\mathfrak{p}) \leq \sqrt{X}$, we have $X/\mathfrak{N}(\mathfrak{p}) \geq \sqrt{X}$, and hence by Theorem 1 of \cite{GL22}, we obtain
\begin{align*}
M_{\ell+1} &\ll \frac{(\log\log X+B_{1})^{\ell-1}}{\ell!} \frac{X}{\log X} \sum_{\mathfrak{N}(\mathfrak{p}^{2}) \leq X} \frac{1}{\mathfrak{N}\mathfrak{p}}
\ll \frac{X}{\log X} \frac{(\log\log X+ B_{0})^{\ell}}{\ell!}.
\end{align*}
\end{proof}
Next we would like to obtain an upper bound for the average value of  $\gamma^{\omega_{\K}(\a)}$ for any positive real $\gamma$
as $\N(\a)$ varies in an interval.
To this end, we recall here a theorem of Shiu which we will apply to bound certain sums in short intervals.
Consider a class $E$ of arithmetic functions $f$ which are non-negative, multiplicative and satisfy
the folllowing two conditions
\begin{enumerate}
\item There exists a positive constant $A$ such that
$$
f(p^{\ell}) \le A_1^{\ell}, ~p~\text{ prime and  }~\ell \ge 1.
$$
\item For every $\beta_1 > 0$ there exists a postiive constant $A_2 =A_2(\beta_1)$ such that
$$
f(n) \le A_2 n^{\beta_1},~n\ge 1.
$$
\end{enumerate}
We now state the theorem of Shiu \cite{Shiu}.

\begin{thm} {\rm (Shiu)} \label{Shiup}
Let $f \in E$, as $X \to \infty$,
$$
\sum_{X-Y \le n \le X} f(n) \ll \frac{Y}{\log X} \exp \left(\sum_{p \le X} \frac{f(p)}{p} \right)
$$
uniformly in $Y$, provided that  $2 \leq X\exp(-\sqrt{\log X}) \leq Y < X$.
\end{thm}
We now make an observation which will be useful in the proof of the following lemma.
By the Chebotarev density theorem, we have
$$
\sum_{\N_{\mathbf{K}/\Q}(\p)\le x, \atop (\p, \mathbf{K}_{\mathfrak{f}}/\mathbf{K}) = 1} 1 = \frac{x}{|H_{\mathfrak{f}}(\K)| \log x} +\rO_{\K}\left(\frac{x}{\log^2 x}\right).
$$
Now by partial summation formula, we get
\begin{eqnarray*}
\sum_{\N_{\mathbf{K}/\Q}(\p)\le x, \atop (\p, \mathbf{K}_{\mathfrak{f}}/\mathbf{K}) = 1} \frac{1}{\N(\p)}
& = &
\sum_{m \le x} \frac{\sum_{\N(\p) = m, (\p, \mathbf{K}_{\mathfrak{f}}/\mathbf{K}) = 1} 1 }{m} \\
& = &
\rO_{\K}\left( \frac{1}{\log x} \right) + \int_2^x  \frac{\sum_{m \le t} \sum_{\N(\p) = m, (\p, \mathbf{K}_{\mathfrak{f}}/\mathbf{K}) = 1} 1 }{t^2}dt\\
& = &
 \int_2^x \frac{ \sum_{\N(\p) \le  t, (\p, \mathbf{K}_{\mathfrak{f}}/\mathbf{K}) = 1} 1 }{t^2}dt + \rO_{\K}\left( \frac{1}{\log x} \right)\\
 & = &
\int_2^x \frac{t }{|H_{\mathfrak{f}}(\K)|t^2 \log t}dt + \rO_{\K} \left( \int_2^x \frac{t }{t^2 \log^2 t}dt \right) +  \rO_{\K}\left( \frac{1}{\log x} \right)\\
\end{eqnarray*}
On computing the above integrals, we have
\begin{eqnarray}\label{mertensum}
\sum_{\N_{\mathbf{K}/\Q}(\p)\le x, \atop (\p, \mathbf{K}_{\mathfrak{f}}/\mathbf{K}) = 1} \frac{1}{\N(\p)} = \frac{\log \log x}{|H_{\mathfrak{f}}(\K)|} + \rO_{\K}(1).
\end{eqnarray}
We now proceed to apply the theorem of Shiu.
\begin{lem}\label{omega asymp in short int}
Let $\gamma \in \mathbb{R}_{>0}$, then
\begin{equation*}
\sum_{X-Y \leq \N(\a) \leq X \atop \a \in \W} \gamma^{\omega_{\K}(\a)} \ll_{\K, \gamma} Y (\log X)^{\frac{\gamma}{|H_{\mathfrak{f}}(\K)|}-1}
\end{equation*}
holds uniformly for $2 \leq X\exp(-\sqrt{\log X}) \leq Y < X$.
\end{lem}

\begin{proof}
Let $f(m) = \sum_{\N(\a) = m, \a \in \W} \gamma^{\omega_{\K}(\a)}$ with the convention that the empty sum is $0$.
We claim that $f$ is multiplicative. This can be seen as follows.
Consider
$$
\sum_{\a \neq(o), \atop{ \a \subset \O_{\K} \atop \a \in \W}} \frac{\gamma^{\omega_{\K}(\a)}}{ \N(\a)^{s}} = \prod_{\p \in \P_{\K}} \left(1 + \frac{\gamma}{\N(\p)^s} \right) = \sum_{n=1}^{\infty} \frac{f(n)}{n^s}.
$$
We now have that $f(n)$ is the
coefficient of $n^s$ in
$
\prod_{p \mid n}  \prod_{\p \mid  p\O_{\K} \atop \p \in \P_{\K}} \left(1 + \frac{\gamma}{\N(\p)^s}  \right).
$
Similarly $f(mn)$ is the
coefficient of $(mn)^s$ in
$$
\prod_{p \mid mn}  \prod_{\p \mid  p\O_{\K}}  \left(1 + \frac{\gamma}{\N(\p)^s}  \right) =
\prod_{p \mid m}  \prod_{\p \mid  p\O_{\K}}  \left(1 + \frac{\gamma}{\N(\p)^s} \right)\prod_{p \mid n}  \prod_{\p \mid  p\O_{\K}}  \left(1 + \frac{\gamma}{\N(\p)^s}  \right).
$$
The last equality follows from the fact that $(m,n)=1$.
This proves the claim.
We have
$$
f(m)  =  \sum_{\N(\a) = m, \a \in \W} \gamma^{\omega_{\K}(\a)} \le  \sum_{\N(\a) = m} \gamma^{n_{\K}\omega(n)} \le  \gamma^{n_{\K}\omega(m)} \tau_{n_{\K}}(m).
$$
For a prime power $m = p^{\ell}$, we have $f(m) \le  \gamma^{n_{\K}} n_{\K}^{\ell}$. For all $m$ and any $\beta_1 >0$, we have
$$
f(m)\le  \gamma^{n_{\K}\omega(m)} \tau_{n_{\K}}(m) \ll 2^{n_{\K}\omega(m) \log_2 \gamma} (\tau(m))^{n_{\K}} \ll_{\beta_1} m^{\beta_1}.
$$
We now apply Shiu's theorem (\thmref{Shiup}), to get
$$
\sum_{X-Y \leq \N\a \leq X \atop \a \in \W} \gamma^{\omega_{\K}(\a)} = \sum_{X-Y \leq m \leq X} f(m) \ll \frac{Y}{\log X} \exp\left( \sum_{p\le X} \frac{f(p)}{p} \right).
$$
Note that
$$
 \sum_{p\le X} \frac{f(p)}{p} =  \sum_{p\le X} \frac{\sum_{\N(\a) = p, \a \in \W} \gamma^{\omega_{\K}(\a)}}{p} =  \sum_{\N(\p) \le X, \atop {\N(\p) \text{ is prime } \atop \p \in \P_{\K}}} \frac{ \gamma}{\N(\p)} \le \sum_{\N(\p) \le X \atop \p\in \P_{\K}} \frac{ \gamma}{\N(\p)}  = \frac{\gamma \log\log X}{|H_{\mathfrak{f}(\K)}|} +\rO_{\K}(\gamma)
$$
where the last step follows from \eqref{mertensum}. This gives us the required lemma.
\end{proof}
We conclude this subsection with the average order of the function which counts the number of ordered factorisations of an integral ideal of $\K$ into exactly $g$ integral ideals.
Let $g \geq 1$ be an integer. For any ideal $\a \subseteq \mathcal{O}_{\K}$, $\tau_{\K,g}(\a)$ denotes the number of ways the ideal $\a$ can be written as an ordered product of $g$ ideals.
For a number field $\K$, we have the following lemma.
\begin{lem}\label{divisor fn lem}
For any positive integer $g \geq 1$, we have
\begin{equation*}
\sum_{\N(\a) \leq x}\tau_{\K,g}(\a)= \alpha_{\K}^{g-1} \frac{x(\log x)^{g-1}}{(g-1)!}+ O_{\K}(x (\log x)^{g-2}).
\end{equation*}
\end{lem}
\begin{proof}
We use the induction hypothesis to prove the claim. It is well-known that (see; \cite{LANGBook})
\begin{equation} \label{norm les x est}
\sum_{\N(\c) \leq x} 1= \alpha_{\K} x+ O(x^{1 - \frac{1}{n_{\K}}}),
\end{equation}
where  $n_{\K}$ is the degree of $\K/\Q$. By using \eqref{norm les x est}, we obtain
\begin{align*}
\sum_{\N(\a) \leq x}\tau_{\K,2}(\a)=& \sum_{\N(\a) \leq x} \sum_{\c \mid \a} 1 = \sum_{\N(\c) \leq x} \sum_{\N(b) \leq \frac{x}{\N(\c)}} 1\\
=& \sum_{\N(\c) \leq x}\bigg(\alpha_{\K} \frac{x}{\N(\c)}+O_{\K}\bigg(\bigg(\frac{x}{\N(\c)}\bigg)^{1-\frac{1}{n_\K}}\bigg)\bigg)\\
=& \alpha_{\K} x \sum_{\N(\c) \leq x} \frac{1}{\N(\c)}+O_{\K}\bigg(x^{1-\frac{1}{n_{\K}}}\sum_{\N(\c) \leq x} \frac{1}{\N(\c)^{1- \frac{1}{n_\K}}}\bigg).
\end{align*}
Also, using \eqref{norm les x est} and partial summation formula it is easy to see that
\begin{align*}
\sum_{1 \leq \N(\c) \leq x} \frac{1}{\N(\c)} =  \alpha_{\K} \log x + \rO_{\K}(1)  \ \ \text{and} \ \sum_{1 \leq \N(\c) \leq x} \frac{1}{\N(\c)^{1- \frac{1}{n_\K}}} \ll_{\K} x^{\frac{1}{n_{\K}}}.
\end{align*}
Thus, we have
$$\sum_{\N(\a) \leq x} \tau_{\K,2}(\a)= \alpha_{\K}^{2} x \log x+ O_{\K}(x).$$
Now, we assume that the claim is true for $\tau_{\K, g-1}$. Therefore, we have
$$\sum_{\N(\c) \leq x}\tau_{\K, g-1}(\c)= \alpha_{\K}^{g-1} \frac{x (\log x)^{g-2}}{(g-2)!}+ O_{\K}(x (\log x)^{g-3}).$$
Since $\tau_{\K, g}(\a)= \sum_{\c \mid \a} \tau_{\K, g-1}(\c)$, we obtain
\begin{equation*}
\sum_{\N(\a) \leq x}\tau_{\K, g}(\a)= \alpha_{\K}^{g-1} \frac{x(\log x)^{g-1}}{(g-1)!}+ O_{\K}(x (\log x)^{g-2}).
\end{equation*}
\end{proof}

\noindent
\subsection{Important Analytic prerequisites}\label{anaprelims2}
The aim of this section is to introduce a few character sum estimates which will be used in the
due course of the article.
We begin by recalling the Large Sieve inequality for number fields due to Wilson (\cite{RW}).
\begin{lem}[\cite{RW}, Wilson]\label{large sieve lem}
Let $\K$ be a number field and $\O_{\K}$ its ring of integers.
Let $\{t_{\a}\}_{\a \subseteq \O_{\K}}$ be a sequence of complex numbers. Let $\chi$ be a character of $H_{\q}(\K)$. Define
$\displaystyle{S_M(\chi)= \sum_{\N(\a) \le M }t_{\a} \chi(\a)}.$
Then
$$\sum_{\N(\mathfrak{q}) \leq Q} \frac{\N(\mathfrak{q})}{\phi(\mathfrak{q})} \psum_{\chi \bmod{\mathfrak{q}}}|S_M(\chi)|^{2} \leq (Q^{2}+M) \sum_{\N(\a) \le M} |t_{\a}|^{2},$$
where $\psum$ denotes summation over primitive characters and $\phi$ is the Euler-Totient function defined on the integral ideals of $\K$ in the following manner
$
\displaystyle{\phi(\q) = \N(\q)\prod_{\p \mid \q}\left(1 - \frac{1}{\N(\p)} \right)}.
$
\end{lem}
\noindent
The following lemma is a generalisation of the Siegel-Walfisz theorem due to L. J. Goldstein \cite{LJG70}.
\begin{lem}\cite{LJG70} \label{sig-wal lem}
Let $\K$ be a normal algebraic number field of finite degree $n_{\K}$ and discriminant $d_{\K}$. Let $\chi$ be a nontrivial generalised Dirichlet character on $H_{\mathfrak{f}}(\K)$. Let $ \epsilon >0$, there exists a positive constant $c=c(\epsilon)$ not depending on $\K$ or $\chi$, such that
$$\sum_{\substack{\N(\p) \leq x \\ (\p, \mathfrak{f})=\mathcal{O}_{\K}}} \chi(\p) \ll Dx (\log^{2}x) \exp(-cn_{\K}(\log x)^{1/2}/D),$$
where $D=n_{\K}^{3}(|d_{\K}| \N(\mathfrak{f}))^{\epsilon}c^{-n_{\K}}$.
\end{lem}
We conclude this section with a generalisation of a lemma of Heilbronn on Generalised Dirichlet characters of Ray class groups.
A key ingredient in the proof is following character sum estimate of Heilbronn.
\begin{lem}\label{char est lem}{\rm (Heilbronn, Lemma 2 of \cite{HB})}
Let $\K$ be an algebraic number field  of discriminant $d_{\K}$ and
degree $n_{\K}$. Let $\chi$ be a non-principal Hecke character defined on  $H_{\b}(\K)$, for an ideal $\b$ with $\N(\b)=b$. Then for any real $x> 1$ and $\epsilon >0$, we have
$$\sum_{\N(\a) \leq x}\chi(\a)= \rO(x^{n_{\K}}|d_{\K}|b)^{\frac{1}{n_{\K}+2}+\epsilon},$$
where the constant implied by the symbol $O$ depends on $n_\K$ and $\epsilon$.
\end{lem}
We now prove our generalisation of Heilbronn's result from \cite{HB}.
\begin{lem}\label{Heilbronn est}
For $i \in \{1, \cdots N\}$ and integral ideals $\c$ with $\N(\c) \le x$, let $a_{i}, b_{\c}$ be complex numbers satisfying $|a_{i}|, |b_{\c}| \leq 1$.
Further let $g$ be any positive integer.
We also assume that we have $N$ distinct Hecke characters $\{ \chi_j\}_{j=1}^N$ modulo ideals of norm $\{\tilde{b}_j\}_{j=1}^N$, respectively.
Then, we have
$$
\sum_{i=1}^{N} \sum_{\substack{ \N(\c) \leq x}}a_{i} b_{\c} \chi_{i}(\c)
\ll_{g} N^{1-\frac{1}{4g}} x \log x + N^{1-\frac{1}{2g}} x^{1-\frac{1}{2(n_{\K}+1)}} (\log x) \bigg(\sum_{i_{1}=1}^N \sum_{\substack{i_{2} = 1 \\ i_{1} \neq i_{2}}}^N(\tilde{b}_{i_1}\tilde{b}_{i_2})^{\frac{2}{n_{\K}+1}}\bigg)^{\frac{1}{4g}}.
$$
\end{lem}
\begin{proof}
By H\"{o}lder's inequality, we have
\begin{align*}
\bigg|\sum_{i=1}^{N} \sum_{\substack{\N(\c) \leq x}}a_{i} b_{\c} \chi_{i}(\c)\bigg| \leq& N^{1-\frac{1}{2g}} \bigg(\sum_{i=1}^{N} \bigg|\sum_{ \N(\c) \le x} b_{\c} \chi_{i}(\c)\bigg|^{2g}\bigg)^{\frac{1}{2g}}\\
\leq &  N^{1-\frac{1}{2g}} \bigg(\sum_{i=1}^{N} \bigg(\sum_{ \N(\c_2), \N(\c_2) \le x} b_{\c_1} \overline{b}_{\c_2} \chi_{i}(\c_1) \overline{\chi}_{i}(\c_2)\bigg)^{g}\bigg)^{\frac{1}{2g}}\\
\leq& N^{1-\frac{1}{2g}} \bigg(\sum_{\N(\e_{1}) \leq x^{g}} f(\e_{1}) \sum_{\N(\e_{2}) \leq x^{g}} \overline{f}(\e_{2}) \sum_{i=1}^{N} \chi_{i}(\e_{1})\bar{\chi}_{i}(\e_{2})\bigg)^{\frac{1}{2g}},
\end{align*}
where
$$
f(\e_{i}) = \sum_{\e_i = \c_{j,1}\cdots \c_{j,g} \atop \N (\c_{j,k}) \le x} b_{c_{j,1}}\cdots b_{c_{j,g}}  \text{ for } j \in \{1,2\}.
$$
We observe that $|f(\e_i)| \le \tau_{\K, g}(\e_i)$. By multiplicativity of $\tau_{\K,g}$ we have
$\tau_{\K, g}(\e_i) = \prod_{\p^s \mid\mid \e_i}  \tau_{\K, g}(\p^{s})$. We now consider
$
 \tau_{\K, g}(\p^{s})^2 \le(s+1)^{2g}  \ll_g \frac{ (s+2g)!}{(2g)! s!}.
$
Since
$
 \tau_{\K, 2g+1}(\p^s) = \frac{ (s+2g)!}{(2g)! s!},
$
we have $  \tau_{\K, g}(\e_i)^2 \ll_g \tau_{\K, 2g+1}(\e_i) $.
 It follows from Lemma \ref{divisor fn lem} that
$$ \sum_{\substack{\c \subseteq \mathcal{O}_{\K} \\ \N(\c) \leq x^{g}}}|f(\e)|^2 \ll_{\K, g} x^{g} (\log x)^{2g}.$$
By applying the Cauchy-Schwarz inequality twice, we get
\begin{eqnarray}\label{tobecont}
\bigg|\sum_{i=1}^{N} \sum_{\substack{ \N(\e) \leq x}}a_{i} b_{\c} \chi_{i}(\e)\bigg|
&\leq &
N^{1-\frac{1}{2g}} \bigg(\sum_{\N(\e_{1}) \leq x^{g}}| f(\e_{1})|^{2}  \sum_{\N(\e_{2}) \leq x^{g}} |f(\e_{2})|^{2}\bigg)^{\frac{1}{4g}}\nonumber\\
& &\times \bigg(\sum_{\N(\e_{1}) \leq x^{g} \atop \N(\e_{2}) \leq x^{g}} \bigg|\sum_{i=1}^{N} \chi_{i}(\e_{1})\bar{\chi}_{i}(\e_{2})\bigg|^{2}\bigg)^{\frac{1}{4g}}\nonumber\\
&\ll & N^{1-\frac{1}{2g}} (x^{2g} (\log x)^{4g})^{\frac{1}{4g}}\bigg(\sum_{i_{1}}\sum_{i_{2}}\sum_{\c_{1}} \sum_{\c_{2}} \chi_{i_{1}}(\c_{1})\bar{\chi}_{i_{1}}(\c_{2}) \chi_{i_{2}}(\c_{2})\bar{\chi}_{i_{2}}(\c_{1})\bigg)^{\frac{1}{4g}}
\end{eqnarray}
If we split the inner sums in \eqref{tobecont} into two, one over the diagonal terms (that is $i_1=i_2$) and otherwise, we get
\begin{align*}
\bigg|\sum_{i=1}^{N} \sum_{\substack{ \N(\c) \leq x}}a_{i} b_{\c} \chi_{i}(\c)\bigg|
&\ll
N^{1-\frac{1}{2g}} (x^{2g} (\log x)^{4g})^{\frac{1}{4g}}\bigg(\sum_{i_{1}} \sum_{\substack{i_{2} }}\bigg|\sum_{\substack{\N(\a) \leq x^{g}}} \chi_{i_{1}} \bar{\chi}_{i_{2}}(\a)\bigg|^{2}\bigg)^{\frac{1}{4g}}\nonumber\\
& \ll
N^{1-\frac{1}{2g}} (x^{2g} (\log x)^{4g})^{\frac{1}{4g}} \bigg( Nx^{2g} + \sum_{i_{1}} \sum_{\substack{i_{2} \\ i_{1} \neq i_{2}}}\bigg|\sum_{\substack{ \N(\a) \leq x^{g}}} \chi_{i_{1}} \bar{\chi}_{i_{2}}(\a)\bigg|^{2}\bigg)^{\frac{1}{4g}}
\end{align*}
Using \lemref{char est lem}, we have
\begin{align*}
\bigg|\sum_{i=1}^{N} \sum_{\substack{ \N(\c) \leq x}}a_{i} b_{\c} \chi_{i}(\c)\bigg|
&\ll N^{1-\frac{1}{4g}} x \log x+ N^{1-\frac{1}{2g}} x^{\frac{1}{2}} \log x\bigg(\sum_{i_{1}} \sum_{\substack{i_{2} \\ i_{1} \neq n_{2}}}\bigg|\sum_{\substack{ \N(\a) \leq x^{g}}} \chi_{i_{1}} \bar{\chi}_{i_{2}}(\a)\bigg|^{2}\bigg)^{\frac{1}{4g}}\nonumber\\
&\ll_{g} N^{1-\frac{1}{4g}} x \log x+ N^{1-\frac{1}{2g}} x^{\frac{1}{2} } \log x\bigg(\sum_{i_{1}} \sum_{\substack{i_{2} \\ i_{1} \neq i_{2}}}(x^{gn_\K}|d_{\K}|\tilde{b}_{i_1}\tilde{b}_{i_2})^{\frac{2}{n_\K+2}+\beta}\bigg)^{\frac{1}{4g}}\nonumber\\
&\ll_{g} N^{1-\frac{1}{4g}} x \log x + N^{1-\frac{1}{2g}} x^{1-\frac{1}{2(n_{\K}+1)}} (\log x)\bigg(\sum_{i_{1}} \sum_{\substack{i_{2} \\ i_{1} \neq i_{2}}}(\tilde{b}_{i_1}\tilde{b}_{i_2})^{\frac{2}{n_{\K}+1}}\bigg)^{\frac{1}{4g}}.
\end{align*}
\end{proof}
\section{Computing the average value of $2^{m \cdot {\rm rk}_4({\rm Cl}_{\L})}$ for $\L \in \mathcal{F}$}\label{avg}
Let $a \in \O_{\K}$. For any ideal $\delta\O_{\K}$ satisfying
 $(\delta\O_{\K}, a\O_{\K}) = \O_{\K}$, we set
\[\chi_{a\O_{\K}}(\delta)=\frac{1}{2^{\omega_\K(a\O_{\K})}}\left(\prod_{\p| a\O_{\K}}\left(\left(\frac{\delta}{\p}\right)+1\right)\right).
\]
Therefore, by \thmref{algresult}, for $\alpha\O_{\K} \in \W$ and $\L = \L_{\alpha} \in \mathcal{F}$, we have
\[
2^{{\rm rk}_4({\rm Cl}_{\L})}\ge \frac{1}{2^{r_\K+2}}\sum_{(a\O_{\K}, b\O_{\K}) \atop \alpha\O_{\K}= ab\O_{\K} }\chi_{a\O_{\K}}(b)\chi_{b\O_{\K}}(a).
\]
We shall use $T(\alpha\O_{\K})$ to denote the sum
$
\displaystyle{\sum_{(a\O_{\K}, b\O_{\K}) \atop \alpha\O_{\K}= ab\O_{\K} }\chi_{a\O_{\K}}(b)\chi_{b\O_{\K}}(a)}.
$
Then
\begin{align*}
T(\alpha\O_{\K})&=\sum_{ab\O_{\K}=\alpha\O_{\K}}\frac{1}{2^{\omega_{\K}(\alpha\O_{\K})}}\prod_{\p|a\O_{\K}}\left(1+\left(\frac{b}{\p}\right)\right)\prod_{\p|b\O_{\K}}\left(1+\left(\frac{a}{\p}\right)\right)\\
&=\frac{1}{2^{\omega_{\K}(\alpha\O_{\K})}}\sum_{ab\O_{\K}=\alpha\O_{\K}}\sum_{c\O_{\K} \mid a\O_{\K}}\left(\frac{b\O_{\K}}{c\O_{\K}}\right)\sum_{d\O_{\K} \mid b\O_{\K}}\left(\frac{a\O_{\K}}{d\O_{\K}}\right).
\end{align*}
Suppose $a\O_{\K}=\partial_0\partial_1$ and $b\O_{\K}=\partial_2\partial_3$, also let
$\partial_0=c\O_{\K}$ and $\partial_3=d\O_{\K}$, then we have
\[
T(\alpha\O_{\K})=\frac{1}{2^{\omega_{\K}(\alpha\O_{\K})}}\sum_{\partial_0\partial_1\partial_2\partial_3=\alpha\O_{\K}}\left(\frac{\partial_2}{\partial_0}\right)\left(\frac{\partial_3}{\partial_0}\right)\left(\frac{\partial_0}{\partial_3}\right)\left(\frac{\partial_1}{\partial_3}\right).
\]
Define
$\Phi_1:\mathbb{F}_2^{2}\times\mathbb{F}_2^{2}\rightarrow\mathbb{F}_2$, by
$
\Phi_1(\bar{u},\bar{v})=(u_1+v_1)(u_1+v_2),
$
where $\bar{u}=(u_1,u_2)$ and $\bar{v}=(v_1,v_2)$.
Note that $\Phi_1(\bar{u},\bar{v})=1$ if and only if
$$(\bar{u},\bar{v}) \in
\{((1,0),(0,0)), ((0,1),(1,1)), ((1,1),(0,0)), ((0,0),(1,1))\}.$$
If we write $\partial_0 =\partial_{00}, \partial_1= \partial_{01}, \partial_2 =\partial_{10}$ and $\partial_3 =\partial_{11}$.
Then it follows that
\[
T(\alpha\O_{\K})=\frac{1}{2^{\omega_{\K}(\alpha\O_{\K})}}\sum_{\alpha\O_{\K}=\partial_{00}\partial_{01}\partial_{10}\partial_{11}}\prod_{\bar{u},\bar{v}\in \mathbb{F}_2^{2}}\left(\frac{\partial_{\bar{u}}}{\partial_{\bar{v}}}\right)^{\Phi_1(\bar{u},\bar{v})}
\]
We interpret the elements $0,1 \in \F_2$ as $0,1 \in \mathbb{N}$ with the convention that $0^0=1$.
Now we consider the $m$-th moment of $2^{{\rm rk}_4({\rm Cl}_{\K})}$, where $m\in \mathbb{N}$
\begin{align}\label{factorisation}
& T_{m}(\alpha\O_{\K}):=\frac{1}{2^{m\omega_{\K}(\alpha\O_{\K})}} \times\\
&\sum_{\substack{\alpha\O_{\K}=\prod_{\bar{u}(1)} \partial_{{\bar{u}(1)}}^{(1)} \\ \vdots \\ \alpha\O_{\K}=\prod_{\bar{u}(m)}\partial_{\bar{u}(m)}^{(m)}}}\prod_{\bar{u}(1),\bar{v}(1)\in \mathbb{F}_2^{2}}\left(\frac{\partial_{{\bar{u}(1)}}^{(1)}}{\partial_{{\bar{v}(1)}}^{(1)}}\right)^{\Phi_1(\bar{u}(1),\bar{v}(1))}\cdots \prod_{\bar{u}(m),\bar{v}(m)\in \mathbb{F}_2^{2}}\left(\frac{\partial_{{\bar{u}(m)}}^{(m)}}{\partial_{{\bar{v}(m)}}^{(m)}}\right)^{\Phi_1(\bar{u}(m),\bar{v}(m))}\nonumber
\end{align}
Suppose that $\alpha \O_{\K}=\prod_{\bar{u}(1)\in \mathbb{F}_2^{2}}\partial_{{\bar{u}(1)}}^{(1)}=\cdots=\prod_{\bar{u}(m)\in \mathbb{F}_2^{2}}\partial_{{\bar{u}(m)}}^{(m)}$
and we define
\[
\partial_{\bar{u}(1),\ldots,\bar{u}(m)}=\gcd \left(\partial_{\bar{u}(1)}^{(1)},\ldots,{\partial_{\bar{u}(m)}^{(m)}}\right)
\]
and
\[
m_{\bar{u}(\ell)}=\prod_{\bar{u}(1)\in\mathbb{F}_2^{2}}\cdots\prod_{\bar{u}(\ell-1)\in\mathbb{F}_2^{2}}\prod_{\bar{u}(\ell+1)\in\mathbb{F}_2^{2}}\cdots\prod_{\bar{u}(m)\in\mathbb{F}_2^{2}}\partial_{\bar{u}(1),\ldots,\bar{u}(m)}.
\]

We claim that $m_{\bar{u}(\ell)}=\partial_{\bar{u}(\ell)}^{(\ell)}$.
Recall that $\alpha\O_{\K}=\prod_{\bar{u}(\ell)\in \mathbb{F}_2^2}\partial_{\bar{u}(\ell)}^{(\ell)}$,
which implies that $\partial_{\bar{u}(\ell)}^{(\ell)}$ is square-free.
If $\p|\partial_{\bar{u}(\ell)}^{(\ell)}$, then it divides $\alpha\O_{\K}$ and this implies that $\p|\partial_{\bar{u}(i)}^{(i)}$, for some $\bar{u}(i)\in \mathbb{F}_2^{2}$, for all $i$.
Hence, $\p|\partial_{\bar{u}(1),\ldots,\bar{u}(\ell),\ldots,\bar{u}(m)}$, for such $\bar{u}(i)$, where $i\neq \ell$ which implies $\p|m_{\bar{u}(\ell)}$, thus $\partial_{\bar{u}(\ell)}^{(\ell)}|m_{\bar{u}(\ell)}$.
Conversely, if $\p|m_{\bar{u}(\ell)}$, then
\[
\p|\gcd\left(\partial_{\bar{u}(1)}^{(1)},\ldots,\partial_{\bar{u}(\ell)}^{(\ell)},\ldots,\partial_{\bar{u}(m)}^{(m)}\right) \text{ for some indices } \bar{u}(i) \in \F_{2}^2 ~\forall~ i \neq \ell \Rightarrow\, \p|\partial_{\bar{u}(\ell)}^{(\ell)}.
\]
Moreover, $m_{\bar{u}(\ell)}$ is square-free. This proves our claim. Therefore we can write $\alpha\O_{\K}$ as
\[
\alpha\O_{\K}=\prod_{\bar{u}(\ell)\in \mathbb{F}_2^2}\partial_{\bar{u}(\ell)}^{(\ell)} =\prod_{\bar{u}(1)\in\mathbb{F}_2^{2}}\cdots\prod_{\bar{u}(\ell)\in\mathbb{F}_2^{2}}\cdots\prod_{\bar{u}(m)\in\mathbb{F}_2^{2}}\partial_{\bar{u}(1),\ldots,\bar{u}(m)}.
\]
Therefore, by replacing $\partial_{\bar{u}(\ell)}^{(\ell)}$ by $m_{\bar{u}(\ell)}$ in \eqref{factorisation}, we get
\[
T_{m}(\alpha\O_{\K})=\frac{1}{2^{m\omega_{\K}(\alpha\O_{\K})}}\sum_{\substack{\alpha\O_{\K}=\prod \partial_{\bar{u}(1),\ldots, \bar{u}(m)} \\ \bar{u}(1),\ldots, \bar{u}(m)\in \mathbb{F}_2^{2}}}
\prod_{\substack{\bar{u}(1),\ldots,\bar{u}(m)\in \mathbb{F}_2^{2} \\ \bar{v}(1),\ldots, \bar{v}(m)\in \mathbb{F}_2^{2}}}\bigg(\frac{\partial_{{\bar{u}(1),\ldots, \bar{u}(m)}}}{\partial_{\bar{v}(1),\ldots, \bar{v}(m)}}\bigg)^{\sum_{i=1}^{m} \Phi_{1}(\bar{u}(i), \bar{v}(i))}.
\]
Therefore, we have
\begin{align*}
\sum_{\L \in \mathcal{F}(X)} 2^{m(rk_{4}(\rm Cl_{\L}))}
\geq& \sum_{\alpha\O_{\K} \in \W(X)}  \frac{1}{2^{m(r_\K+2)}} T_{m}(\alpha\O_{\K})  ~\geq~  \frac{1}{2^{m(r_\K+2)}} N.
\end{align*}
where
\begin{equation}
N= \sum_{ \alpha\O_{\K} \in \W(X)} T_{m}(\alpha\O_{\K}).
\end{equation}
Now, we need to estimate $N$. Let $\bar{u}, \bar{v} \in \mathbb{F}_{2}^{2m}$ with $\bar{u}=(\bar{u}(1), \dots, \bar{u}(m))$ and $\bar{v}=(\bar{v}(1), \dots, \bar{v}(m))$. We define
\begin{equation}\label{phi-m defn}
\Phi_{m}(\bar{u}, \bar{v})= \sum_{i=1}^{m} \Phi_{1}(\bar{u}(i), \bar{v}(i)).
\end{equation}
Therefore we have proved the following theorem.
\begin{thm}\label{simpl notn thm}
We have
$$N= \sum_{(\partial_{\bar{u}})_{\bar{u} \in \F_{2}^{2m}} \in \mathcal{D}(X,m)} \frac{1}{2^{m \omega_{\K}(\prod_{\bar{u} \in \F_{2}^{2m}}\partial_{\bar{u}})}} \prod_{\bar{u}, \bar{v}} \bigg(\frac{\partial_{\bar{u}}}{\partial_{\bar{v}}}\bigg)^{\Phi_{m}(\bar{u}, \bar{v})},$$
where $\mathcal{D}(X,m)$ is the set of $4^m$ -tuples of squarefree and coprime ideals $\partial_{\bar{u}}$ such that
\begin{enumerate}
\item the index $
\bar{u}=(\bar{u}(1), \dots, \bar{u}(m)) \in \mathbb{F}_{2}^{2m}$ and
\item  $ \prod_{\bar{u} \in \mathbb{F}_{2}^{2m}} \partial_{\bar{u}} \in \W(X)$.
\end{enumerate}
\end{thm}

\subsection{Eliminating indices corresponding to a large number of prime divisors}\label{largenumofdiv}
For any ideal $\mathfrak{a} \subseteq \O_{\K}$,
$\tau_{\K, m}(\mathfrak{a})$ denotes the number of ordered ways of writing $\mathfrak{a}$ as a product of $m$ ideals. Let
$$
\sideset{}{_1}\sum = \sum_{(\partial_{\bar{u}})_{\bar{u} \in \F_{2}^{2m}} \in \mathcal{D}(X,m) \atop \omega_{\K}(\prod_{\bar{u} \in \F_{2}^{2m}}\partial_{\bar{u}}) > \Omega } \frac{1}{2^{m \omega_{\K}(\prod_{\bar{u} \in \F_{2}^{2m}}\partial_{\bar{u}})}} \prod_{\bar{u}, \bar{v}} \bigg(\frac{\partial_{\bar{u}}}{\partial_{\bar{v}}}\bigg)^{\Phi_{m}(\bar{u}, \bar{v})}.
$$
 Here, the parameter $\Omega$ will be chosen later. Since $\tau_{\K, 4^m}(\mathfrak{a})=2^{2m\omega_{\K}(\mathfrak{a})}$ for any squarefree ideal $\mathfrak{a}$, we have
\begin{align*}
\sideset{}{_1}\sum \ll \sum_{\substack{\mathfrak{N}(\mathfrak{a}) \leq X \\ \omega_{\K}(\mathfrak{a}) > \Omega}} \mu^{2}(\mathfrak{a}) \frac{\tau_{\K, 4^m}(\mathfrak{a})}{2^{m\omega_{\K}(\mathfrak{a})}} \ll \sum_{\substack{\mathfrak{N}(\mathfrak{a}) \leq X \\ \omega_{\K}(\mathfrak{a}) > \Omega}} \mu^{2}(\mathfrak{a}) 2^{m\omega_{\K}(\mathfrak{a})}.
\end{align*}
By using Lemma \ref{l prime fact est} and Stirling's formula, we write
\begin{align*}
\sideset{}{_1}\sum \ll& \sum_{v > \Omega} 2^{mv} \sum_{\substack{\mathfrak{N}(\mathfrak{a}) \leq X \\ \omega_{\K}(\mathfrak{a})=v}} \mu^{2}(\mathfrak{a})
\ll_{\K} \sum_{v > \Omega} 2^{mv} \frac{X}{\log X} \frac{(\log \log X+B_{0})^{v-1}}{(v-1)!}\\
\ll_{\K,m}& \sum_{v \geq \Omega} 2^{mv} \frac{X}{\log X} \frac{(\log \log X+B_{1})^{v}}{v!} \ll_{\K,m} \sum_{v \geq \Omega}\frac{(2^{m}(\log \log X+B_{1}))^{v}}{(v/e)^{v}} \cdot \frac{X}{\log X}.
\end{align*}
By choosing $\Omega= e 4^{m}(\log\log X+B_1)$, we see that the sum
$$\sum_{v \geq \Omega}\frac{(2^{m}(\log \log X+B_{1}))^{v}}{(v/e)^{v}} \leq \sum_{v \geq \Omega} \frac{1}{2^{mv}}= {\rm O}(1),$$
for any $m \geq 1$ and hence we have
\begin{equation}\label{larg factors est}
\sideset{}{_1}\sum \ll_{\K,m} \frac{X}{\log X}.
\end{equation}
\subsection{Dissection of the range of the variables $\partial_{\bar{u}}$ into sub-intervals}\label{fam4}
Let $\Delta= 1+ \log ^{-2^{m}}X$ and $A_{\bar{u}}=\Delta^{r}$, for some positive integer $r$, for all $\bar{u} \in \mathbb{F}_{2}^{2m}$. For $A=(A_{\bar{u}})_{\bar{u} \in \mathbb{F}_{2}^{2m}}$, we define
$$T(X, m, A)= \sum_{(\partial_{\bar{u}})}\bigg(\prod_{\bar{u}}2^{-m\omega_{\K}(\partial_{\bar{u}})} \bigg) \prod_{\bar{u}, \bar{v}} \bigg(\frac{\partial_{\bar{u}}}{\partial_{\bar{v}}}\bigg)^{\Phi_{m}(\bar{u}, \bar{v})},$$
where the sum is over $\partial_{\bar{u}} \in \mathcal{D}(X,m)$ such that $A_{\bar{u}} \leq \N(\partial_{\bar{u}}) \leq \Delta A_{\bar{u}}$, $\omega_{\K}( \prod_{\bar{u} \in \mathbb{F}_2^{2m}}\partial_{\bar{u}}) \leq \Omega$ for all $\bar{u} \in \mathbb{F}_{2}^{2m}$.
Here $\mathcal{D}(X,m)$ is defined as in Theorem \ref{simpl notn thm}. By using \eqref{larg factors est}, we write
\begin{equation}\label{range subdivid eq}
N = \sum_{\alpha\O_{\K}\in \W(X)} T_{m}(\alpha\O_{\K})= \sum_{A}T(X, m, A) + {\rm O}\bigg(\frac{X}{\log X}\bigg),
\end{equation}
where $A$ is such that $\prod_{\bar{u} \in \mathbb{F}_{2}^{2m}} A_{\bar{u}} \leq X$.
We will now consider 4 families of tuples $(A_{\bar{u}})_{\bar{u} \in \mathbb{F}_{2}^{2m}}$ and show that their contribution is negligible.

\noindent
\textbf{First family: } The first family is defined by:
$(A_{\bar{u}})$ such that $\prod_{\bar{u} \in \mathbb{F}_{2}^{2m}} A_{\bar{u}} \geq \Delta^{-4^{m}} X$. We have
\begin{align*}
\sum_{\substack{A \\ \prod_{\bar{u} \in \mathbb{F}_{2}^{2m}} A_{\bar{u}} \geq \Delta^{-4^{m}} X}} |T(X, m, A)| \ll \sum_{\Delta^{-4^{m}} X \leq \N(\mathfrak{a}) \leq X \atop \a \in \W}  \mu^{2}(\mathfrak{a})2^{m \omega_{\K}(\mathfrak{a})}.
\end{align*}
Using Lemma \ref{omega asymp in short int},
$\displaystyle{\sum_{\Delta^{-4^{m}} X \leq \N(\mathfrak{a}) \leq X \atop \a \in \W} \mu^{2}(\mathfrak{a})2^{m \omega_{\K}(\mathfrak{a})}
\ll_{\K,\mathfrak{f},m} X(1-\Delta^{-4^{m}}) (\log X)^{2^{m}-1}}$.

\noindent
We note that
$\Delta^{-4^{m}}=(1+\log ^{-2^{m}}X)^{-4^{m}}= 1-4^{m} \log^{-2^{m}}X+{\rm O}_m(\log^{-2^{m+1}}X)$; hence
\begin{align*}
\sum_{\substack{A \\ \prod_{\bar{u} \in \mathbb{F}_{2}^{2m}} A_{\bar{u}} \geq \Delta^{-4^{m}} X}} |T(X, m, A)| \ll_{\K, m} \frac{X}{\log X}.
\end{align*}
\begin{rmk}
Note that if $\prod_{\bar{u} \in \mathbb{F}_{2}^{2m}} A_{\bar{u}} \leq \Delta^{-4^{m}} X$ then $\N\big(\prod_{\bar{u} \in \mathbb{F}_{2}^{2m}} \partial_{\bar{u}}\big) \leq \Delta^{4^{m}} \prod_{\bar{u} \in \mathbb{F}_{2}^{2m}} A_{\bar{u}} \leq X$.
\end{rmk}
To introduce the other families, we define
\begin{align}\label{dag defn}
&X^{\dagger}= (\log X)^{\max( 20, 10(n_{\K}+1))(2+4^{m}(1+2^{m}))},~~~ \eta(m)= 2^{-m}\beta\\
&X^{\ddagger} \text{ is the least } \Delta^{\ell} \text{ such that } \Delta^{\ell} \geq \exp(\log ^{\eta(m)}X),\label{ddag defn}
\end{align}
where $\beta > 0$ is sufficiently small.

\vspace{1mm}
\noindent
\textbf{Second family: } This consists of $(A_{\bar{u}})$ such that $\prod_{\bar{u} \in \mathbb{F}_{2}^{2m}} A_{\bar{u}} < \Delta^{-4^{m}} X$ and
\begin{equation}\label{2nd family}
\text{ At most } 2^{m}-1 \text{ of the } A_{\bar{u}} \text{ in } A=(A_{\bar{u}})_{\bar{u} \in \mathbb{F}_{2}^{2m}} \text{ are greater than } X^{\ddagger}.
\end{equation}
Then
\begin{align*}
\sum_{A \text{ satisfies \eqref{2nd family} }} |T(X, m, A)| \leq \sum_{0 \leq r \leq 2^{m}-1} &\sum_{\N(\tilde{\partial}_{1}) \leq (X^{\ddagger})^{4^{m}-r}} \mu^{2}(\tilde{\partial}_{1}) \tau_{\K, 4^{m}-r}(\tilde{\partial}_{1})2^{-m\omega_{\K}(\tilde{\partial}_{1})}\\
&\times \sum_{\N(\tilde{\partial}_{2}) \leq \frac{X}{\N(\tilde{\partial}_{1})} \atop \tilde{\partial}_{2} \in \W} \mu^{2}(\tilde{\partial}_{2}) \tau_{\K, r}(\tilde{\partial}_{2})2^{-m\omega_{\K}(\tilde{\partial}_{2})}.
\end{align*}
By using Lemma \ref{omega asymp in short int}, we bound the inner sum by
\begin{align*}
\sum_{\N(\tilde{\partial}_{2}) \leq \frac{X}{\N(\tilde{\partial}_{1})}\atop \tilde{\partial}_{2} \in \W} \mu^{2}(\tilde{\partial}_{2}) \tau_{\K, r}(\tilde{\partial}_{2})2^{-m\omega_{\K}(\tilde{\partial}_{2})}
 \ll& \sum_{\N(\tilde{\partial}_{2}) \leq \frac{X}{\N(\tilde{\partial}_{1})} \atop \tilde{\partial}_{2} \in \W} \tau_{\K,2}(\tilde{\partial}_{2})^{-m+\log_{2}r}\\
\ll_{\K,m}& \frac{X}{\N(\tilde{\partial}_{1})} (\log X)^{\frac{r2^{-m}}{{|H_{\mathfrak{f}}(\K)|}}-1}.
\end{align*}
Thus
\begin{align*}
\sum_{A \text{ satisfies \eqref{2nd family} }} |T(X, m, A)| \ll_{\K,m}& X \sum_{0 \leq r \leq 2^{m}-1} (\log X)^{\frac{r2^{-m}}{{|H_{\mathfrak{f}}(\K)|}}-1} \sum_{\N(\tilde{\partial}_{1})
\ll_{\K,m} (X^{\ddagger})^{4^{m}-r}} \mu^{2}(\tilde{\partial}_{1}) \frac{2^{m\omega_{\K}(\tilde{\partial}_{1})}}{\N(\tilde{\partial_{1}})}\\
\ll_{\K,m}& X \sum_{0 \leq r \leq 2^{m}-1} (\log X)^{ \frac{r2^{-m}}{{|H_{\mathfrak{f}}(\K)|}}-1} \prod_{\substack{\mathfrak{p} \\ \N(\mathfrak{p}) \leq (X^{\ddagger})^{4^{m}-r}}} \bigg(1+\frac{2^{m}}{\N(\mathfrak{p})}\bigg)\\
\ll_{\K,m} &  X \sum_{0 \leq r \leq 2^{m}-1} (\log X)^{\frac{r2^{-m}}{{|H_{\mathfrak{f}}(\K)|}}-1} \prod_{\substack{\mathfrak{p} \\ \N(\mathfrak{p}) \leq (X^{\ddagger})^{4^{m}-r}}} \bigg(1-\frac{1}{\N(\mathfrak{p})}\bigg)^{-2^{m}}.
\end{align*}
By using Mertens's theorem for number fields (Theorem 1, \cite{GL22}), we have
\begin{align*}
\sum_{A \text{ satisfies \eqref{2nd family} }} |T(X, m, A)| \ll_{\K,m}& \frac{X}{\log X} \bigg(\frac{(\log X)^{\frac{2^{-m}}{{|H_{\mathfrak{f}}(\K)|}}\cdot2^{m}}-1}{(\log X)^{\frac{2^{-m}}{{|H_{\mathfrak{f}}(\K)|}}}-1}\bigg) (\log X^{\ddagger})^{2^{m}} \\
\ll_{\K,m}& X \log^{2^{m}\eta(m) - 1 + \frac{1-2^{-m}}{{|H_{\mathfrak{f}}(\K)|}}} X.
\end{align*}

\begin{defn}
The variables $\bar{u}$ and $\bar{v}$ are said to be linked if
$\Phi_{m}(\bar{u}, \bar{v})+\Phi_{m}(\bar{v}, \bar{u})=1$.
\end{defn}

\vspace{1mm}
\noindent
\textbf{Third family: } This consists of $(A_{\bar{u}})$ such that
\begin{enumerate}[(i)]
\item $\prod_{\bar{u}} A_{\bar{u}} < \Delta^{-4^{m}}X$ \label{3rd family i}
\item At least $2^{m}$ of the $A_{\bar{u}}$ in $A=(A_{\bar{u}})_{\bar{u} \in \mathbb{F}_{2}^{2m}}$ are greater than $X^{\ddagger}$. \label{3rd family ii}
\item There exists two linked indices $\bar{u}, \bar{v}$ such that $A_{\bar{u}}, A_{\bar{v}} \geq X^{\dagger}$.\label{3rd family iii}
\end{enumerate}
Without loss of generality, we assume $\Phi_m(\bar{u}, \bar{v})=1$.
\begin{align}\label{sums3}
|T(X, m, A)|
\ll &
\sum_{\substack{(\partial_{\bar{w}})_{\bar{w} \neq \bar{v}, \bar{u}} \\ A_{\bar{w}} \leq \N(\partial_{\bar{w}}) \leq \Delta A_{\bar{w}}}}  \prod_{\bar{w} \neq \bar{u}, \bar{v}}  2^{-m \omega_{\K}(\partial_{\bar{w}})}\\
&
\times
\bigg|\sum_{A_{\bar{u}} \leq \N(\partial_{\bar{u}}) \leq \Delta A_{\bar{u}}}\sum_{A_{\bar{v}} \leq \N(\partial_{\bar{v}}) \leq \Delta A_{\bar{v}} \atop \partial_{\bar{v}} \in \W}a(\partial_{\bar{u}},(\partial_{\bar{w}})_{\bar{w} \neq \bar{v}, \bar{u}}) a(\partial_{\bar{v}},(\partial_{\bar{w}})_{\bar{w} \neq \bar{v}, \bar{u}}) \bigg(\frac{\partial_{\bar{u}}}{\partial_{\bar{v}}}\bigg)\bigg|,\nonumber
\end{align}
where $\displaystyle{a(\partial_{\bar{u}}, (\partial_{\bar{w}})_{\bar{w} \neq \bar{v}, \bar{u}})= 2^{-m \omega_{\K}(\partial_{\bar{u}})}\prod_{\bar{w} \neq \bar{u}, \bar{v}} \bigg(\frac{\partial_{\bar{w}}}{\partial_{\bar{u}}}\bigg)^{\Phi_{m}(\bar{w}, \bar{u})} \prod_{\bar{w} \neq \bar{u}, \bar{v}} \bigg(\frac{\partial_{\bar{u}}}{\partial_{\bar{w}}}\bigg)^{\Phi_{m}(\bar{u}, \bar{w})}}$.

Analogously we have $a(\partial_{\bar{v}}, (\partial_{\bar{w}})_{\bar{w} \neq \bar{v}, \bar{u}})$.
We would now like to apply \lemref{Heilbronn est}.
For this we note that $\left( \frac{.}{\partial_{\bar{v}}} \right)$ is a primitive character modulo $\partial_{\bar{v}}$.
Each character appearing in \eqref{sums3} is primitive with a distinct conductor. Hence these characters are distinct. We now apply \lemref{Heilbronn est}, to obtain, for any integer $g >0$
\begin{align*}
|T(X, m, A)| \ll& \prod_{\bar{w} \neq \bar{u}, \bar{v}} A_{\bar{w}}\bigg( A_{\bar{v}} A_{\bar{u}}(A_{\bar{v}}^{\frac{-1}{4g}} \log A_{\bar{u}})+A_{\bar{v}}^{1-\frac{1}{2g}} A_{\bar{u}}^{1-\frac{1}{2(n_{\K}+1)}} \log A_{\bar{u}} \bigg(\sum_{\N(\a) \leq A_{\bar{v}}}\N(\a)^{\frac{2}{n_{\K}+1}}\bigg)^{\frac{1}{2g}}\bigg)\\
\ll& \prod_{\bar{w} \neq \bar{u}, \bar{v}} A_{\bar{w}}\bigg( A_{\bar{v}} A_{\bar{u}}(A_{\bar{v}}^{\frac{-1}{4g}} \log A_{\bar{u}})+A_{\bar{v}} A_{\bar{u}}^{1-\frac{1}{2(n_{\K}+1)}} A_{\bar{v}}^{\frac{1}{g(n_{\K}+1)}} \log A_{\bar{u}}\bigg)\\
\ll& X((X^{\dagger})^{\frac{-1}{4g}} \log X)+ X (A_{\bar{u}}^{\frac{-1}{2(n_{\K}+1)}} A_{\bar{v}}^{\frac{1}{g(n_{\K}+1)}})\log X .
\end{align*}
By choosing $g=5$, we see that if $A_{\bar{v}} \ll A_{\bar{u}}^{2}$ holds, then
$$|T(X, m, A)| \ll X (X^{\dagger})^{- \min \left( \frac{1}{10(n_{\K}+1)}, \frac{1}{20} \right)} \log X.
$$

We now tackle the case $A_{\bar{u}}^{2} \ll A_{\bar{v}}$. Using the Cauchy Schwarz inequality, we may bound the sum inside the absolute value in \eqref{sums3} by
$$ \ll A_{\bar{u}}^{\frac{1}{2}} \bigg(\sum_{\partial_{\bar{u}}}\bigg|\sum_{\partial_{\bar{v}}} a(\partial_{\bar{v}},(\partial_{\bar{w}})_{\bar{w} \neq \bar{v}, \bar{u}}) \bigg(\frac{\partial_{\bar{u}}}{\partial_{\bar{v}}}\bigg)\bigg|^{2}\bigg)^{\frac{1}{2}}.$$
If $\partial_{\bar{u}} = s\O_{\K}$, then by definition of $\left(\frac{.}{\partial_{\bar{v}}}\right)$ for $\partial_{\bar{v}}  \in \W(X)$,
$
 \bigg(\frac{\partial_{\bar{u}}}{\partial_{\bar{v}}}\bigg) = \bigg(\frac{s\varepsilon_s}{\partial_{\bar{v}}}\bigg).
 $
 Here $\varepsilon_s$ is as defined before \lemref{primitivity rmk}.
  By Lemma \ref{primitivity rmk}, $\left(\frac{s\epsilon_s}{.}\right)$
 is a primitive character modulo $4\partial_{\bar{u}}$.
Now, by Lemma \ref{large sieve lem}, we obtain (when $A_{\bar{u}}^{2} \ll A_{\bar{v}}$)
\begin{align*}
|T(X, m, A)| \ll& \prod_{\bar{w} \neq \bar{u}, \bar{v}} A_{\bar{w}} ( A_{\bar{u}}^{\frac{1}{2}} A_{\bar{v}}^{\frac{1}{2}}(A_{\bar{u}}^{2}+A_{\bar{v}})^{\frac{1}{2}}) \ll X A_{\bar{u}}^{\frac{-1}{2}} \ll X (X^{\dagger})^{\frac{-1}{2}}.
\end{align*}
By summing over all $O((\log X)^{4^{m}(1+2^{m})})$ possible $A$ and using the
definition of $X^{\dagger}$ we get
$$\sum_{A \text{ satisfies \eqref{3rd family i}, \eqref{3rd family ii}, \eqref{3rd family iii}}} |T(X, m, A)| \ll \frac{X}{\log X}.$$

\vspace{1mm}
\noindent
\textbf{Fourth family: }This consists of $(A_{\bar{u}})$ such that
\begin{enumerate}[(i)]
\item $\prod_{\bar{u} \in \mathbb{F}_{2}^{2m}} A_{\bar{u}} < \Delta^{-4^{m}}X$ \label{4th family i}
\item There exists two linked indices $\bar{u}, \bar{v}$ such that $2 \leq A_{\bar{v}} < X^{\dagger}$ and $A_{\bar{u}}  \geq X^{\ddagger}$\label{4th family ii}
\item $\omega_{\K}(\prod_{\bar{u} \in \mathbb{F}_{2}^{2m}} \partial_{\bar{u}}) \leq \Omega$.
\item Two indices $\bar{u}, \bar{v}$ with $A_{\bar{u}}, A_{\bar{v}} > X^{\dagger}$ are always unlinked.
\item At least $2^{m}$ of the $A_{\bar{u}}$ in $A=(A_{\bar{u}})_{\bar{u} \in \mathbb{F}_{2}^{2m}}$ are greater than $X^{\ddagger}$.
\end{enumerate}
By assumption \eqref{4th family ii}, there exists an index $\bar{v}$ which is linked to $\bar{u}$.
We observe that there may be more than one index linked to $\bar{u}$. Recall that
$$T(X, m, A)= \sum_{(\partial_{\bar{u}})}\bigg(\prod_{\bar{u}}2^{-m\omega_{\K}(\partial_{\bar{u}})} \bigg) \prod_{\bar{u}, \bar{v}} \bigg(\frac{\partial_{\bar{u}}}{\partial_{\bar{v}}}\bigg)^{\Phi_{m}(\bar{u}, \bar{v})}.$$
Let $I_1 \subset \mathbb{F}_2^{2m}$ be the set of all indices linked with $\bar{u}$
and let $I_2 \subset \mathbb{F}_2^{2m}$ be the set of all indices unlinked with $\bar{u}$.
We now divide the above sum into three sums as follows
$$T(X, m, A)= \sum_{(\partial_{\bar{w}})_{\bar{w} \in I_1}}\sum_{(\partial_{\bar{w}_1})_{\bar{w}_1 \in I_2}} \sum_{\partial_{\bar{u}}} \bigg(\prod_{\bar{u}_1 \in \mathbb{F}_2^{2m}}2^{-m\omega_{\K}(\partial_{\bar{u}_1})} \bigg) \prod_{\bar{u}_2, \bar{v} \in \mathbb{F}_2^{2m}} \bigg(\frac{\partial_{\bar{u}_2}}{\partial_{\bar{v}}}\bigg)^{\Phi_{m}(\bar{u}_2, \bar{v})}.$$
Taking absolute values, we get
$$|T(X, m, A)|\le \sum_{(\partial_{\bar{w}})_{\bar{w} \in I_1}}\sum_{(\partial_{\bar{w}_1})_{\bar{w}_1 \in I_2}} \bigg|\sum_{\partial_{\bar{u}}} \bigg(\prod_{\bar{u}_1 \in \mathbb{F}_2^{2m}}2^{-m\omega_{\K}(\partial_{\bar{u}_1})} \bigg) \prod_{\bar{u}_2, \bar{v} \in \mathbb{F}_2^{2m}} \bigg(\frac{\partial_{\bar{u}_2}}{\partial_{\bar{v}}}\bigg)^{\Phi_{m}(\bar{u}_2, \bar{v})}\bigg|.$$
Any term given by $ \bigg(\frac{\partial_{\bar{u}_2}}{\partial_{\bar{v}}}\bigg)$ with $\bar{u}_2, \bar{v} \neq \bar{u}$, can be pulled out of the sum over $\partial_{\bar{u}}$ and can be bounded by $1$.
If we have a term given by $ \bigg(\frac{\partial_{\bar{u}}}{\partial_{\bar{v}}}\bigg)$ with $\bar{v} \in I_2$, then the term $ \bigg(\frac{\partial_{\bar{v}}}{\partial_{\bar{u}}}\bigg)$
will also appear in the product and by \lemref{quadreciprocity}, they can be multiplied to give $1$. Finally we are left with terms given by $ \bigg(\frac{\partial_{\bar{u}}}{\partial_{\bar{v}}}\bigg)$ or $ \bigg(\frac{\partial_{\bar{v}}}{\partial_{\bar{u}}}\bigg)$ with $\bar{v} \in I_1$.
By \lemref{quadreciprocity} and multiplicativity of these characters, we get
\begin{align*}
|T(X,m,A)| \leq&\sum_{(\partial_{\bar{w}})_{\bar{w} \in I_1}}\sum_{(\partial_{\bar{w}_1})_{\bar{w}_1 \in I_2}}\bigg|\sum_{\partial_{\bar{u}}}\frac{\mu^{2}(\prod \partial_{\bar{w}})}{2^{m\omega_{\K(\partial_{\bar{u}})}}} \bigg(\frac{\partial_{\bar{u}}}{\prod_{\bar{w}_1 \in I_2}\partial_{\bar{w}_1}}\bigg) \bigg|\\
\leq& \sum_{(\partial_{\bar{w}})_{\bar{w} \in I_1}}\sum_{(\partial_{\bar{w}_1})_{\bar{w}_1 \in I_2}} \sum_{0 \leq \ell \leq \Omega} \frac{1}{2^{m\ell}}\bigg|\sum_{\substack{\partial_{\bar{u}} \\ \omega(\partial_{\bar{u}})=\ell}} \mu^{2}(\prod \partial_{\bar{w}}) \bigg(\frac{\partial_{\bar{u}}}{\prod_{\bar{w}_1 \in I_2}\partial_{\bar{w}_1}}\bigg) \bigg|.
\end{align*}
Writing $\partial_{\bar{u}}= \p_{1} \cdots \p_{\ell}$ in ascending order of absolute norm, the inner sum is bounded by
\begin{equation}\label{sumpl}
\bigg|\sum_{\substack{\partial_{\bar{u}} \\ \omega(\partial_{\bar{u}})=\ell}} \mu^{2}(\prod \partial_{\bar{w}}) \bigg(\frac{\partial_{\bar{u}}}{\prod_{\bar{w}_1 \in I_2}\partial_{\bar{w}_1}}\bigg) \bigg| \leq \sum_{\p_{1}} \sum_{\p_{2}} \dots \bigg|\sum_{\p_{\ell}} \mu^{2}(\p_{1} \dots \p_{\ell}\prod_{\bar{w}\neq \bar{u}} \partial_{\bar{w}}) \bigg(\frac{\p_{\ell}}{\prod_{\bar{w}_1 \in I_2}\partial_{\bar{w}_1}}\bigg)\bigg|,
\end{equation}
where $\bigg(\frac{\cdot}{\prod_{\bar{w}_1 \in I_2}\partial_{\bar{w}_1}}\bigg)$ is non-trivial generalised Dirichlet character.
We note here that
\begin{equation}\label{pbound}
A_{\bar{u}}^{1/\ell} \le \N(\p_{\ell}) \le \frac{\Delta A_{\bar{u}}}{\N(\p_1\cdots \p_{\ell-1})}.
\end{equation}
 If $\p_{\ell}\mid  \prod_{\bar{w}\neq \bar{u}} \partial_{\bar{w}}$ then the corresponding term in \eqref{sumpl} is zero and number of such $\p_{\ell}$ is at most $\Omega$ by condition (iii) above. Therefore, we have
\begin{eqnarray*}
&&\sum_{\N(\p_{\ell}) \leq \frac{\Delta A_{\bar{u}}}{\N(\p_{1}, \dots, \p_{\ell-1})}} \mu^{2}(\p_{1} \dots \p_{\ell}\prod_{\bar{w}\neq \bar{u}} \partial_{\bar{w}}) \bigg(\frac{\p_{\ell}}{\prod_{\bar{w}_1 \in I_2}\partial_{\bar{w}_1}}\bigg) \\
&=&
\mu^{2}(\p_{1} \dots \p_{\ell-1}\prod_{\bar{w}\neq \bar{u}}\partial_{\bar{w}}) \sum_{\N(\p_{\ell}) \leq \frac{\Delta A_{\bar{u}}}{\N(\p_{1}, \dots, \p_{\ell-1})} \atop \gcd(\p_\ell, \prod_{\bar{w} \neq \bar{u}}\partial_{\bar{w}}) = \O_{\K} } \bigg(\frac{\p_{\ell}}{\prod_{\bar{w}_1 \in I_2}\partial_{\bar{w}_1}}\bigg)
\end{eqnarray*}
However
\begin{eqnarray*}
&&\mu^{2}(\p_{1} \dots \p_{\ell-1}\prod_{\bar{w}\neq \bar{u}}\partial_{\bar{w}}) \sum_{\N(\p_{\ell}) \leq \frac{\Delta A_{\bar{u}}}{\N(\p_{1}, \dots, \p_{\ell-1})} \atop \gcd(\p_\ell, \prod_{\bar{w} \neq \bar{u}}\partial_{\bar{w}}) = \O_{\K} } \bigg(\frac{\p_{\ell}}{\prod_{\bar{w}_1 \in I_2}\partial_{\bar{w}_1}}\bigg)\\
&=&
\mu^{2}(\p_{1} \dots \p_{\ell-1}\prod_{\bar{w}\neq \bar{u}}\partial_{\bar{w}}) \sum_{\N(\p_{\ell}) \leq \frac{\Delta A_{\bar{u}}}{\N(\p_{1}, \dots, \p_{\ell-1})} \atop \gcd(\p_\ell, \prod_{\bar{w}_1 \in I_2}\partial_{\bar{w}_1}) = \O_{\K} } \bigg(\frac{\p_{\ell}}{\prod_{\bar{w}_1 \in I_2}\partial_{\bar{w}_1}}\bigg) + \rO(\Omega).
\end{eqnarray*}
Applying Lemma \ref{sig-wal lem}, we obtain for any $\beta_1 > 0$,
\begin{align*}
\sum_{\N(\p_{\ell}) \leq \frac{\Delta A_{\bar{u}}}{\N(\p_{1}, \dots, \p_{\ell-1})} \atop \gcd(\p_{\ell}, \prod_{\bar{w}_1 \in I_2}\partial_{\bar{w}_1}) = \O_{\K}} &\bigg(\frac{\p_{\ell}}{\prod_{\bar{w}_1 \in I_2}\partial_{\bar{w}_1}}\bigg) \ll\\
\left(\N \left(\prod_{\bar{w}_1 \in I_2}\partial_{\bar{w}_1}\right)\right)^{\beta_1}& \frac{\Delta A_{\bar{u}}}{\N(\p_{1} \dots \p_{\ell-1})} \log^{2}\frac{\Delta A_{\bar{u}}}{\N(\p_{1} \dots \p_{\ell-1})} \exp\bigg(-2c_{\K,\epsilon}\frac{\bigg(\log \frac{\Delta A_{\bar{u}}}{\N(\p_{1} \dots \p_{\ell-1})}\bigg)^{\frac12}}{(\N(\prod_{\bar{w}_1 \in I_2} \partial_{\bar{w}_1}))^{\beta_1}}\bigg)\\
\ll_{\beta_1}&
\frac{\N(\prod_{\bar{w}_1 \in I_2}\partial_{\bar{w}_1})^{\beta_1(1+B)}}{\log^{B-2}\frac{\Delta A_{\bar{u}}}{\N(\p_{1} \dots \p_{\ell-1})}} \cdot \frac{\Delta A_{\bar{u}}}{\N(\p_{1} \dots \p_{\ell-1})}.
\end{align*}
Choosing $\beta_1= \frac{1}{2(1+B)}$, we obtain
\begin{align*}
\sum_{\N(\p_{\ell}) \leq \frac{\Delta A_{\bar{u}}}{\N(\p_{1}, \dots, \p_{\ell-1}) }\atop \gcd(\p_{\ell}, \prod_{\bar{w}_1 \in I_2}\partial_{\bar{w}_1}) = \O_{\K}} \bigg(\frac{\p_{\ell}}{\prod_{\bar{w}_1 \in I_2}\partial_{\bar{w}_1}}\bigg) \ll_{B}& \frac{\N(\prod_{\bar{w}_1 \in I_2}\partial_{\bar{w}_1})^{1/2}}{\log^{B-2}\frac{\Delta A_{\bar{u}}}{\N(\p_{1} \dots \p_{\ell-1})}} \cdot \frac{\Delta A_{\bar{u}}}{\N(\p_{1} \dots \p_{\ell-1})}\\
\end{align*}
By \eqref{pbound} we have
$$\sum_{\N(\p_{\ell}) \leq \frac{\Delta A_{\bar{u}}}{\N(\p_{1}, \dots, \p_{\ell-1})}\atop \gcd(\p_{\ell}, \prod_{\bar{w}_1 \in I_2}\partial_{\bar{w}_1}) = \O_{\K}} \bigg(\frac{\p_{\ell}}{\prod_{\bar{w}_1 \in I_2}\partial_{\bar{w}_1}}\bigg) \ll_{B} \frac{\N(\prod_{\bar{w}_1 \in I_2}\partial_{\bar{w}_1})^{1/2}}{\log^{B-2}A_{\bar{u}}^{1/\ell}} \cdot \frac{\Delta A_{\bar{u}}}{\N(\p_{1} \dots \p_{\ell-1})}.$$
By the definition of the fourth family, we have $A_{\bar{u}} \geq X^{\ddagger} \geq \exp(\log^{\eta(m)}X)$. This implies that
\begin{align*}
\frac{\log^{B-2}A_{\bar{u}}}{\ell^{B-2}} \gg_{\K, B} \frac{\log^{(B-2)\eta(m)}X}{(\log\log X)^{B-2}} \gg_{\K, B} \log ^{\frac{(B-2) \eta(m)}{2}}X.
\end{align*}
Therefore, we have
\begin{equation*}
\sum_{\N(\p_{\ell}) \leq \frac{\Delta A_{\bar{u}}}{\N(\p_{1}, \dots, \p_{\ell-1})}\atop \gcd(\p_{\ell}, \prod_{\bar{w}_1 \in I_2}\partial_{\bar{w}_1}) = \O_{\K}} \bigg(\frac{\p_{l}}{\prod_{\bar{w}_1 \in I_2}\partial_{\bar{w}_1}}\bigg) \ll_{B}\Delta A_{\bar{u}} \frac{\N(\prod_{\bar{w}_1 \in I_2}\partial_{\bar{w}_1})^{1/2}}{\N(\p_{1} \dots \p_{l-1})} \log^{-\frac{(B-2)\eta(m)}{2}}X.
\end{equation*}
Also, we see that
$$\sum_{\ell \leq \Omega} \frac{1}{2^{m\ell}} \sum_{\N(\p_{1}\cdots \p_{\ell-1}) \leq \Delta A_{\bar{u}}} \frac{\mu^{2}(\p_{1} \dots \p_{\ell-1})}{\N(\p_{1}\cdots \p_{\ell-1})} \ll \log\Delta A_{\bar{u}} \ll \log X$$
and
$
\sum_{\N(\partial_{\bar{w}_1}) \le X^{\dagger}}   \N(\partial_{\bar{w}_1})^{\frac12} \ll (X^{\dagger})^{\frac32}.
$
Further
$$
 \sum_{(\partial_{\bar{w}})_{\bar{w} \in I_1}}\sum_{(\partial_{\bar{w}_1})_{\bar{w}_1 \in I_2}} \sum_{0 \leq \ell \leq \Omega} \frac{1}{2^{m\ell}} \sum_{\p_{1}} \dots \sum_{\p_{\ell -1}, \atop \N(\p_1\cdots \p_{\ell-1}) \le \Delta A_{\bar{u}}^{1 -1/\ell}} 1
 \ll
 X(X^{\ddagger})^{-\frac{1}{\Omega}}.
$$
Hence, for $B \gg 1$, we have
\begin{align*}
|T(X,m,A)| \ll_{B}&  \prod_{\bar{w}\in I_1} A_{\bar{w}}\cdot \prod_{\bar{w}_1 \in I_2} (X^{\dagger})^{\frac32} \cdot ((\log X)^{-B\eta(m)+1} A_{\bar{u}}) +\Omega X(X^{\ddagger})^{-\frac{1}{\Omega}}\\
\ll_{B}& X \bigg((\log X)^{-B\eta(m)+1} \cdot  (X^{\dagger})^{\frac{3\cdot 4^m}{2}}+ \frac{\log \log X}{(X^{\ddagger})^{1/\Omega} }\bigg) \ll_{B} \frac{X}{\log X}.
\end{align*}
Combining the estimates for all the four families and recalling $\eta(m) = 2^{-m}\beta$ we get
\begin{prop}\label{4family prop}
For every $m \geq 1$ and $\beta > 0$ sufficiently small, we have
\begin{equation}
\sum_{\alpha\O_{\K} \in \W(X)} T_{m}(\alpha\O_{\K}) = \psum_{A} T(X, m, A)+O\left(\frac{X}{(\log X)^{1- \frac{1}{{|H_{\mathfrak{f}}(\K)|}} }} (\log X )^{2^{m}\eta(m) - \frac{2^{-m}}{{|H_{\mathfrak{f}}(\K)|}}}\right),
\end{equation}
where the sum is over tuples $(A_{\bar{u}})_{\bar{u} \in \mathbb{F}_2^{2m}}$ satisfying the following conditions:
\begin{enumerate}[(i)]
\item $\prod_{\bar{u} \in \mathbb{F}_{2}^{2m}} A_{\bar{u}} < \Delta^{-4^{m}}X$
\item At least $2^{m}$ indices satisfy $A_{\bar{u}} > X^{\ddagger}$
\item Two indices $\bar{u}, \bar{v}$ with $A_{\bar{u}}, A_{\bar{v}} \ge X^{\dagger}$ are always unlinked
\item If two indices $\bar{u}, \bar{v}$ with $A_{\bar{v}} \leq A_{\bar{u}}$ are linked then either $A_{\bar{v}}=1$ or  $2 \leq A_{\bar{v}} < X^{\dagger}$ and $ A_{\bar{u}}  \leq X^{\ddagger}$.
\end{enumerate}
\end{prop}
\subsection{Geometry of unlinked indices}
\begin{lem}(Lemma 18, \cite{FK07})\label{unlinked}
Let $m \geq 1$ be an integer and let $\mathcal{U} \subseteq \mathbb{F}_{2}^{2m}$ be a set of unlinked indices. Then $\# \mathcal{U} \leq 2^{m}$ and for any $\bar{u} \in \mathbb{F}_{2}^{2m}$, $\bar{u}+\mathcal{U}$ is also a set of unlinked indices. If $\#\mathcal{U}=2^{m}$ then $\mathcal{U}$ is a vector subspace of dimension $m$ in $\mathbb{F}^{2m}_{2}$ or a coset of such a subspace.
\end{lem}
\begin{prop}\label{gui}
For every $m \geq 1$ and $\beta > 0$, we have
\begin{equation*}
\sum_{ \alpha  \O_{\K} \in \W(X)} T_{m}(\alpha\O_{\K}) = \psum_{A} T(X, m, A)+O\left(\frac{X}{(\log X)^{1- \frac{1}{{|H_{\mathfrak{f}(\K)}|}} }} (\log X)^{2^{m}\eta(m) - \frac{2^{-m}}{{|H_{\mathfrak{f}(\K)}|}}} \right),
\end{equation*}
where the $(*)$ in the sum is used to indicate that $A$ varies over
\begin{enumerate}[(i)]
\item $\prod_{\bar{u} \in \mathbb{F}_{2}^{2m}} A_{\bar{u}} < \Delta^{-4^{m}}X$,
\item $\mathcal{U}= \{\bar{u} : A_{\bar{u}} > X^{\ddagger}\}$ is a maximal set of unlinked indices,
\item $A_{\bar{u}}=1$ for $\bar{u} \notin \mathcal{U}$.
\end{enumerate}
\end{prop}
\begin{proof}
For $B \gg_{m,\beta} 1$  we have $X^{\ddagger} > (\log^{\eta(m)}X)^{B} > X^{\dagger}$.
Therefore, by part (iii) of Proposition \ref{4family prop},~ $\mathcal{U}$ is a set of unlinked indices. Further by part (ii) of Proposition \ref{4family prop},~
$\mathcal{U}$ contains at least $2^m$ elements. By \lemref{unlinked}, $\mathcal {U}$ must be a maximal set of unlinked indices.
For $\bar{v} \notin \mathcal{U}$ and any $\bar{u} \in \mathcal{U}$, $\bar{u}$ and $\bar{v}$ are linked. Since $A_{\bar{u}} \geq A_{\bar{v}}$, by part (iv) of Proposition \ref{4family prop} either $A_{\bar{v}}=1$ or  $2 \leq A_{\bar{v}} < X^{\dagger}$ and $A_{\bar{v}} \leq A_{\bar{u}}  \leq X^{\ddagger}$. But $A_{\bar{u}} > X^{\ddagger}$ so $A_{\bar{v}}=1$.
\end{proof}
\begin{defn}\label{admi}
Let $\mathcal{U} \subseteq \mathbb{F}_{2}^{2m}$ denote an unlinked set of $2^{m}$ indices. We say $A=(A_{\bar{u}})_{\bar{u} \in \mathbb{F}_{2}^{2m}}$ is admissible for $\mathcal{U}$ if it satisfies:
\begin{enumerate}[(i)]
\item $\prod_{\bar{u} \in \mathbb{F}_{2}^{2m}} A_{\bar{u}} < \Delta^{-4^{m}}X$,
\item $\mathcal{U}= \{\bar{u} : A_{\bar{u}} > X^{\ddagger}\}$,
\item $A_{\bar{u}}=1$ for $\bar{u} \notin \mathcal{U}$.
\end{enumerate}
\end{defn}

\noindent
\textbf{Note:} If $A_{\bar{u}}=1$ then $\partial_{\bar{u}}=\mathcal{O}_{\K}$.
\subsection{The final estimate}\label{final}
We begin by recalling the sum we want to estimate.
Recall that
\begin{align*}
\sum_{ \alpha\O_{\K} \in \W(X)} T_{m}(\alpha\O_{\K}) =&\sum_{\substack{ (\partial_{\bar{u}}) \\ \prod \partial_{\bar{u}} \in \W(X)}}\frac{1}{2^{m\omega_{\K}(\prod \partial_{\bar{u}})}} \prod_{\bar{u}, \bar{v}} \bigg(\frac{\partial_{\bar{u}}}{\partial_{\bar{v}}}\bigg)^{\Phi_{\K}(\bar{u}, \bar{v})}
\end{align*}
By Proposition \ref{gui}, we have
\begin{align*}
\sum_{ \alpha\O_{\K} \in \W(X)} T_{m}(\alpha\O_{\K}) =&\sum_{\substack{ (\partial_{\bar{u}}) \\ \prod \partial_{\bar{u}} \in \W(X)}}\frac{1}{2^{m\omega_{\K}(\prod \partial_{\bar{u}})}} \prod_{\bar{u}, \bar{v}} \bigg(\frac{\partial_{\bar{u}}}{\partial_{\bar{v}}}\bigg)^{\Phi_{\K}(\bar{u}, \bar{v})}\\
=& \psum_{A} T(X, m, A)+O\left(\frac{X}{(\log X)^{1- \frac{1}{{|H_{\mathfrak{f}(\K)}|}} }} (\log X)^{2^{m}\eta(m) - \frac{2^{-m}}{{|H_{\mathfrak{f}(\K)}|}}} \right),
\end{align*}
where the $(*)$ in the sum is used to indicate that $A$ varies over
\begin{enumerate}[(i)]
\item $\prod_{\bar{u} \in \mathbb{F}_{2}^{2m}} A_{\bar{u}} < \Delta^{-4^{m}}X$,
\item $\mathcal{U}= \{\bar{u} : A_{\bar{u}} > X^{\ddagger}\}$ is a maximal set of unlinked indices,
\item $A_{\bar{u}}=1$ for $\bar{u} \notin \mathcal{U}$.
\end{enumerate}
By the definition of admissible $A$ (Definition \ref{admi}), we get
\begin{align*}
\sum_{ \alpha\O_{\K} \in \W(X)} T_{m}(\alpha\O_{\K})
=& \psum_{A} T(X, m, A)+O\left(\frac{X}{(\log X)^{1- \frac{1}{{|H_{\mathfrak{f}(\K)}|}} }} (\log X)^{2^{m}\eta(m) - \frac{2^{-m}}{{|H_{\mathfrak{f}(\K)}|}}} \right)\\
=& \sum_{\mathcal{U}} \sum_{A \ \text{admissible for} \ \mathcal{U}} T(X, m, A)+O\left(\frac{X}{\log^{1- \frac{1}{{|H_{\mathfrak{f}(\K)}|}} }X} \log^{2^{m}\eta(m) - \frac{2^{-m}}{{|H_{\mathfrak{f}(\K)}|}}} X\right).
\end{align*}
Here
\begin{equation*}
T(X, m, A)=\sum_{\substack{ (\partial_{\bar{u}}),  \omega_\K(\prod \partial_{\bar{u}}) \leq \Omega \\ A_{\bar{u}} \leq \mathfrak{N}(\partial_{\bar{u}}) \leq \Delta A_{\bar{u}}\\  \prod \partial_{\bar{u}} \in \W(X) }}\frac{1}{2^{m\omega_{\K}(\prod \partial_{\bar{u}})}} \prod_{\bar{u}, \bar{v}} \bigg(\frac{\partial_{\bar{u}}}{\partial_{\bar{v}}}\bigg)^{\Phi_{\K}(\bar{u}, \bar{v})}.
\end{equation*}
Let us look at
\begin{eqnarray*}
\bigg|  \sum_{A} \sum_{\substack{ (\partial_{\bar{u}}),  \omega_\K(\prod \partial_{\bar{u}}) > \Omega \\ A_{\bar{u}} \leq \mathfrak{N}(\partial_{\bar{u}}) \leq \Delta A_{\bar{u}}\\  \prod \partial_{\bar{u}} \in \W(X) }}\frac{1}{2^{m\omega_{\K}(\prod \partial_{\bar{u}})}} \prod_{\bar{u}, \bar{v}} \bigg(\frac{\partial_{\bar{u}}}{\partial_{\bar{v}}}\bigg)^{\Phi_{\K}(\bar{u}, \bar{v})} \bigg|
& \le& \sum_{A} \sum_{\substack{ (\partial_{\bar{u}}),  \omega_\K(\prod \partial_{\bar{u}}) > \Omega \\ A_{\bar{u}} \leq \mathfrak{N}(\partial_{\bar{u}}) \leq \Delta A_{\bar{u}}\\  \prod \partial_{\bar{u}} \in \W(X) }}\frac{1}{2^{m\omega_{\K}(\prod \partial_{\bar{u}})}} \\
\end{eqnarray*}
Since
\begin{align*}
\sum_{A} \sum_{\substack{ (\partial_{\bar{u}}),  \omega_\K(\prod \partial_{\bar{u}}) > \Omega \\ A_{\bar{u}} \leq \mathfrak{N}(\partial_{\bar{u}}) \leq \Delta A_{\bar{u}}\\  \prod \partial_{\bar{u}} \in \W(X) }}\frac{1}{2^{m\omega_{\K}(\prod \partial_{\bar{u}})}}
 \ll \sum_{\ell \geq \Omega} \sum_{\substack{\N(\a) \leq X \\ \omega_{\K}(\a)=\ell}} 2^{-m\omega_{\K}(\a)} \tau_{4^{m}}(\a) \ll \frac{X}{\log X}. (\text{cf. subsection \ref{largenumofdiv}})
\end{align*}
For a set of maximally unlinked indiced $\mathcal{U}$ and an $A$ admissible for $\mathcal{U}$, we set
\begin{eqnarray*}
T'(X, m, A) & = & \sum_{\substack{(\partial_{\bar{u}})_{\bar{u} \in \mathcal{U}} \\ A_{\bar{u}} \leq \mathfrak{N}(\partial_{\bar{u}}) \leq \Delta A_{\bar{u}}\\  \prod_{\bar{u} \in \mathcal{U}} \partial_{\bar{u}} \in \W(X)}} \mu^{2}(\prod_{\bar{u} \in \mathcal{U}} \partial_{\bar{u}}) 2^{-m\omega_{\K}(\prod_{\bar{u} \in \mathcal{U}} \partial_{\bar{u}})}\prod_{\bar{u}, \bar{v} \in \mathcal{U}} \bigg(\frac{\partial_{\bar{u}}}{\partial_{\bar{v}}}\bigg)^{\Phi_{\K}(\bar{u}, \bar{v})}.
\end{eqnarray*}
By \lemref{quadreciprocity} we get
\begin{eqnarray*}
T'(X, m, A) & = & \sum_{\substack{(\partial_{\bar{u}})_{\bar{u} \in \mathcal{U}} \\ A_{\bar{u}} \leq \mathfrak{N}(\partial_{\bar{u}}) \leq \Delta A_{\bar{u}}\\  \prod_{\bar{u} \in \mathcal{U}} \partial_{\bar{u}} \in \W(X)}} \mu^{2}(\prod_{\bar{u} \in \mathcal{U}} \partial_{\bar{u}}) 2^{-m\omega_{\K}(\prod_{\bar{u} \in \mathcal{U}} \partial_{\bar{u}})}.
\end{eqnarray*}
Therefore we have
$$ \sum_{ \alpha\O_{\K} \in \W(X)} T_{m}(\alpha\O_{\K}) =\sum_{\mathcal{U}} \sum_{A  \text{ admissible for }  \mathcal{U}} T'(X, m, A)+O\left(\frac{X}{\log^{1- \frac{1}{{|H_{\mathfrak{f}(\K)}|}} }X} (\log X)^{2^{m}\eta(m) - \frac{2^{-m}}{{|H_{\mathfrak{f}(\K)}|}}} \right).$$
We now consider the sum
$$
\sum_{A \ \text{admissible for} \ \mathcal{U}}\sum_{\substack{(\partial_{\bar{u}})_{\bar{u} \in \mathcal{U}} \\ A_{\bar{u}} \leq \mathfrak{N}(\partial_{\bar{u}}) \leq \Delta A_{\bar{u}}\\  \prod_{\bar{u} \in \mathcal{U}} \partial_{\bar{u}} \in \W(X)}} \mu^{2}(\prod_{\bar{u} \in \mathcal{U}} \partial_{\bar{u}}) 2^{-m\omega_{\K}(\prod_{\bar{u} \in \mathcal{U}} \partial_{\bar{u}})}.
$$
Since $A$ is admissible for $\mathcal{U}$, at least $2^m$ entries of the tuple $A$ are $1$ and these are all the entries with $A_{\bar{u}} \le X^{\ddagger}$.  We first note that given any $\a \in \W(\Delta^{-4^m}X)$ with norm greater than $X^{\ddagger}$ appears as a product, $\prod_{\bar{u} \in \mathcal{U}} \partial_{\bar{u}}$, for some  $A$ admissible for $\mathcal{U}$.
Each such $\a$ appears as many times as the number of factorisations of $\a$ into
 $2^m$ ideals each of norm greater than $X^{\ddagger}$.
Therefore we have
\begin{align*}
\sum_{A \ \text{admissible for} \ \mathcal{U}} T'(X, m, A)=& \sum_{A \ \text{admissible for} \ \mathcal{U}}\sum_{\substack{(\partial_{\bar{u}})_{\bar{u} \in \mathcal{U}} \\ A_{\bar{u}} \leq \mathfrak{N}(\partial_{\bar{u}}) \leq \Delta A_{\bar{u}}\\  \prod \partial_{\bar{u}} \in \W(X)}} \mu^{2}(\prod_{\bar{u} \in \mathcal{U}} \partial_{\bar{u}}) 2^{-m\omega_{\K}(\prod \partial_{\bar{u}})}\\
=& \sum_{ \a \in \W(X) } \mu^{2}(\a) \tau_{\K,2^{m}}(\a) 2^{-m\omega_{\K}(\a)}+O\bigg( \sum_{\Delta^{-4^{m}}X \leq \N(\a) \leq X \atop \a \in \W(X)} \mu^{2}(\a) \tau_{\K,2^{m}}(\a)2^{-m\omega_{\K}(\a)}\bigg)\\
&+O\bigg(\sum_{ \b \in \W(X^{\ddagger})} 2^{-m\omega_{\K}(\b)} \tau_{\K, 2^m}(\b) \mu^{2}(\b)\sum_{\c \in \W(\frac{X}{\N(b)})} 2^{-m\omega_{\K}(\c)} (2^{m}-1)^{\omega_{\K}(\c)} \mu^{2}(\c)\bigg).
\end{align*}
where the last term corresponds to the summands which have at least one factor that is less than $X^{\ddagger}$. By using Lemma \ref{omega asymp in short int}, the inner sum of the second error term is bounded by
\begin{align*}
I=&\sum_{ \c \in \W(\frac{X}{\N(b)})} 2^{-m\omega_{\K}(\c)} (2^{m}-1)^{\omega_{\K}(\c)} \mu^{2}(\c)
= \sum_{ \c \in \W(\frac{X}{\N(b)})} \tau_{\K,2}(\c)^{-m+\log_{2}(2^{m}-1)}\\
\ll& \frac{X}{\N(\b)} (\log X)^{\frac{2^{-m+\log_{2}(2^{m}-1)}}{{|H_{\mathfrak{f}(\K)}|}}-1}
\ll   \frac{X}{\N(\b)} (\log X)^{\frac{1- 2^{-m}}{{|H_{\mathfrak{f}(\K)}|}}-1}.
\end{align*}
Then by using Mertens's theorem for number fields (Theorem 1, \cite{GL22}), we can bound the second  $\rO$ term as
\begin{align*}
X (\log X)^{\frac{1- 2^{-m}}{{|H_{\mathfrak{f}(\K)}|}}-1} \sum_{1 \leq \N(\b) \leq X^{\ddagger}} \frac{ \mu^{2}(\b)}{{\N(\b)}}\ll& X (\log X)^{\frac{1- 2^{-m}}{{|H_{\mathfrak{f}(\K)}|}}-1} \prod_{\N(\p) \leq X^{\ddagger}}\bigg(1+\frac{1}{\N(\p)}\bigg)\\
\ll& \frac{X}{\log^{1- \frac{1}{{|H_{\mathfrak{f}(\K)}|}} }X} (\log X)^{2^{m}\eta(m) - \frac{2^{-m}}{{|H_{\mathfrak{f}(\K)}|}}}
\end{align*}
By using similar arguments as in the estimate of the first family, we see that the first error term is bounded by $X/\log X$. Therefore, we have
\begin{equation*}
\sum_{A \ \text{admissible for} \ \mathcal{U}} T'(X, m, A)=\sum_{\a \in \W(X)} \mu^{2}(\a) \tau_{\K, 2^{m}}(\a) 2^{-m\omega_{\K}(\a)}+O\left(\frac{X}{\log^{1- \frac{1}{{|H_{\mathfrak{f}(\K)}|}} }X} (\log X)^{2^{m}\eta(m) - \frac{2^{-m}}{{|H_{\mathfrak{f}(\K)}|}}} \right).
\end{equation*}
As noted before for a squarefree ideal $\a$, $\tau_{\K, 2^{m}}(\a) = 2^{m\omega_{\K}(\a)}$. This lets us conclude that

\begin{eqnarray*}
\frac{1}{2^{m(r_\K+2)}}\sum_{ \alpha\O_{\K} \in \W(X)} T_{m}(\alpha\O_{\K}) & = & \frac{1}{2^{m(r_\K+2)}} \sum_{\mathcal{U}} \sum_{\a \in \W(X)} 1+O\left(\frac{X}{\log^{1- \frac{1}{{|H_{\mathfrak{f}(\K)}|}} }X} (\log X)^{2^{m}\eta(m) - \frac{2^{-m}}{{|H_{\mathfrak{f}(\K)}|}}} \right)\\
& = & \frac{\mathcal{N}( 2m,2)}{2^{m(r_\K+1)}}\sum_{\a \in \W(X)} 1+O\left(\frac{X}{\log^{1- \frac{1}{{|H_{\mathfrak{f}(\K)}|}} }X} (\log X)^{2^{m}\eta(m) - \frac{2^{-m}}{{|H_{\mathfrak{f}(\K)}|}}} \right).
\end{eqnarray*}
\section{Examples}\label{examples}
In this section, we give examples of number fields which satisfy all the  above conditions.
\subsection{Example 1:} $\K =\Q(i).$\\
It is obvious that $\Q(i)$ satisfies conditions \ref{one} and \ref{two}.
We know that $2\Z$ ramifies in $\K$.
For the unique prime $\p \mid 2\O_{\K}$, $ |(\O_{\K}/\p^2)^*| = \N(\p)(\N(\p)-1) = 2$.
Since $1 \not \equiv i \bmod \p^2$, condition \ref{units} is also satisfied.
\subsection{Example 2:} $\K =\Q(\sqrt{3}).$\\
Again, $\K$ is known to satisfy conditions \ref{one} and \ref{two}.
In this case $\O_{\K}^* = \{\pm 1\} \times \langle \sqrt{3} - 2 \rangle$.
Again we note that $2\Z$ ramifies in $\K$.
For the unique prime $\p \mid 2\O_{\K}$, $ |(\O_{\K}/\p^2)^*| = \N(\p)(\N(\p)-1) = 2$. In $\O_{\K}$, we have
$$
\p^2 = 2\O_{\K} = (1-\sqrt{3})\O_{\K} \cdot (1+\sqrt{3})\O_{\K} = (1-\sqrt{3})^2 \O_{\K},
$$
this shows that $2-\sqrt{3} \not\equiv 1 \bmod \p^2$. Hence we have condition \ref{units}.
\subsection{Other quadratic examples:} Some other examples of fields for which the above argument can be applied are
$\Q(\sqrt{d})$ for
$$
d \in \{7, 11, 19, 23, 27, 31, 43, 47, 59, 63, 67, 71, 75, 83, 99 \}.
$$
These were generated using SAGE and many more can be generated in this fashion.

\subsection{Example 3 (higher degree):} Let $\K$ be a Galois number field with class number $1$.
If $2\Z$ splits in $\K$, then for any prime $\p \mid 2\O_{\K}$, we know that $2 \notin \p^2$.
This is because the valuation of $2\O_{\K}$ with respect to $\p$ is $1$.
Therefore $1 \not\equiv -1 \bmod \p^2$.
This gives us condition \ref{units} above.
Here are some examples of such cubic fields.

\begin{table}[ht]
\caption{Examples of Galois cubic fields with class number $1$ in which $2\Z$ splits}

\centering
\begin{tabular}{c c c c }
\hline\hline
Serial &  Defining polynomial of the field & Serial &  Defining polynomial of the field  \\
\hline
 1 & $x^3 + x^2 - 10x - 8$ &  8 &$ x^3 + x^2 - 94x + 304 $\\
 2 &$x^3 + x^2 - 14x + 8 $&  9 & $x^3 + x^2 - 102x - 216$\\
 3 & $x^3 + x^2 - 36x - 4 $&10 & $x^3 + x^2 - 144x - 16 $\\
 4 & $x^3 + x^2 - 42x + 80$ & 11 & $x^3 + x^2 - 146x - 504 $\\
 5  & $x^3 + x^2 - 52x + 64 $&12 & $x^3 + x^2 - 152x - 220 $\\
 6 & $x^3 + x^2 - 74x - 256 $&13 &$x^3 + x^2 - 166x + 536 $\\
 7 & $x^3 + x^2 - 76x - 212 $& 14 &$x^3 + x^2 - 200x + 512 $\\
 \hline
\end{tabular}
\end{table}

\bibliographystyle{plain}  

\end{document}